\documentclass[12pt]{amsart}
\usepackage[utf8]{inputenc}
\usepackage[margin=1in]{geometry}

\usepackage{changes}
\usepackage{comment}
\usetikzlibrary{patterns}
\usepackage{accents}

\usepackage{fullpage}
\usepackage{amssymb,amsfonts}
\usepackage[all,arc]{xy}
\usepackage{enumerate}
\usepackage{enumitem}
\usepackage{mathrsfs}
\usepackage{mathtools}
\usepackage{tikz}
\usepackage{tikz-cd}
\usepackage{graphicx}
\usepackage[font={small}]{caption}
\usepackage{float}
\usepackage{subcaption}
\usepackage{hyperref}
\usepackage{rotating}
\usepackage{changes}
\usepackage{standalone}
\usepackage{changes}
\usepackage{relsize}

\usepackage{changes}

\newtheorem{thm}{Theorem}[section]
\newtheorem*{thm*}{Theorem}
\newtheorem{cor}[thm]{Corollary}
\newtheorem{prop}[thm]{Proposition}
\newtheorem{lem}[thm]{Lemma}

\theoremstyle{definition}
\newtheorem{defn}[thm]{Definition}

\theoremstyle{remark}
\newtheorem{rem}[thm]{Remark}

\title{Thompson's group $F$, tangles, and link homology}

\begin{document}

\begin{abstract}
We extend a construction of Jones to associate $(n, n)$-tangles with elements of Thompson's group $F$ and prove that it is asymptotically faithful as $n \to\infty$. Using this construction we show that the oriented Thompson group $\vec F$ admits a lax group action on a category of Khovanov's chain complexes. 
\end{abstract}

\author{Vyacheslav Krushkal}

\author{Louisa Liles}

\author{Yangxiao Luo}

\address{Department of Mathematics, University of Virginia, Charlottesville VA 22904}
\email{krushkal@virginia.edu, lml2tb@virginia.edu, yl8by@virginia.edu}

\maketitle

\section{Introduction}
Motivated by questions in conformal field theory, Vaughan Jones initiated the study of Thompson's groups $F$ of certain piecewise linear homeomorphisms of the interval in the context of quantum topology. As part of this program, Jones constructed unitary representations of $F$ and showed how to associate a link in $3$-space with any element of the group \cite{jones14, Jones19}. 
An analogue of the Alexander theorem \cite[Theorem 5.3.1]{jones14} states that any link type arises from this construction.
Jones suggested that the Thompson group can serve a role similar to that of braid groups that have been classically used to represent links. It is worth noting that $F$ is a single group generating all links, as opposed to a collection $B_n$ of braid groups, $n\geq 1$.

Using the theory of strand diagrams developed in \cite{belk_thesis, matucci}, we extend the Jones construction to associate an $(n,n)$-tangle $ T_n(g)$, for certain values of $n$, with an element $g$ of Thompson's group $F$. We show that this construction is asymptotically faithful:
\begin{thm} \label{thm: intro aymptotic faithfulness}
    Suppose $g, h \in F$ are two distinct elements. Then for sufficiently large $n$,   the tangles $T_n(g), T_n(h)$ are not isotopic.
\end{thm}

This result contrasts the original construction where different group elements may give rise to the same link. In fact, finding a set of ``Markov moves'' relating any two such group elements is an open problem \cite[Remark 5.3.3 (2)]{jones14}. The construction of tangles is given in Section \ref{sec: tangles}, and a more detailed version of the above result is stated as Theorem \ref{thm: aymptotic faithfulness for unoriented F}.

It was shown in \cite{jones14, Aiello} that all oriented links arise from the ``oriented'' Thompson group $\vec F$, a certain subgroup of $F$ constructed by Jones. This notion was extended in \cite{HOMFLY} where coefficients of unitary
representations of $\vec F$ were analyzed in the context of oriented topological quantum field theory (TQFT) invariants.
We define oriented strand diagrams and use them to associate oriented tangles with elements of $\vec F$, see Section \ref{sec: orientability} for details.  

Using these methods, 
we construct an action of Thompson's group $F$ on a certain category of Khovanov's chain complexes. 
Our result may be seen as an instance of the general framework of group actions on triangulated categories. Such actions of braids groups, and more generally of mapping class groups, are an expected feature of $(3+1)$-dimensional TQFTs; see for example the discussion and references in~\cite[Section 6.5]{khovanov}.
Recall the graded ring $H^n$, constructed in \cite{khovanov} using the Temperley-Lieb $2$-category. Khovanov defined a braid group action on $\mathcal{K}^{m}_{n}$, the homotopy category of chain complexes of geometric $(H^m, H^n)$-bimodules \cite{khovanov}. Our main result, Theorem \ref{thm: action} in Section \ref{sec: Khovanov}, may be thought of as an analogue for Thompson's group. In this case the action is lax, see Definition \ref{defn: lax group action}, due to the properties of the construction of links and tangles from elements of the group $F$. This notion was also considered in \cite[Definition 8.4]{LOT}, under the name of a {\em weak action}. Theorem 15 in \cite{LOT} constructed a weak action of mapping class groups in the context of bimodules in the bordered Heegaard Floer theory.

{\em Acknowledgements.} The authors were supported in part by NSF grant DMS-2105467. Louisa Liles was also supported in part by NSF RTG grant DMS-1839968.

\section{Background on Thompson's group $F$ and strand diagrams}\label{sec:strand diagrams}
This section summarizes the relevant background on Thompson's group $F$ and its correspondence with reduced $(n,n)$-strand diagrams. For more details, we direct the reader to \cite[Chapters $1$ and $7$]{belk_thesis} and \cite{canon}.
The Thompson group $F$ consists of piecewise linear orientation-preserving self-homeomorphisms of the unit interval $[0,1]$ such that all derivatives are powers of $2$ and all points of non-differentiability occur at dyadic numbers, that is, numbers of the form $\frac{a}{2^{b}}$ for $a, b \in \mathbb{Z}$. 

One way to specify an element of the Thompson group is with a pair of standard dyadic partitions, which are partitions of the unit interval in which every sub-interval is of the form $[\frac{a}{2^{b}},\frac{a+1}{2^{b}}]$. Given an ordered pair $(I,J)$ of standard dyadic partitions with the same number of sub-intervals, there is a unique element $g \in F$ such that $g(I)=J$. Standard dyadic partitions correspond bijectively to planar, binary, rooted trees, so two such trees also determine an element of $F$. It is convenient to represent the pair of trees $T^{}_I, T^{}_J$ by drawing the union of $T^{}_I$ and the vertical reflection of $T^{}_J$, with their leaves identified, as shown on the right. (This will be seen as a special case of the product of strand diagrams below.)

\begin{figure}[ht]
\begin{center}
\includegraphics{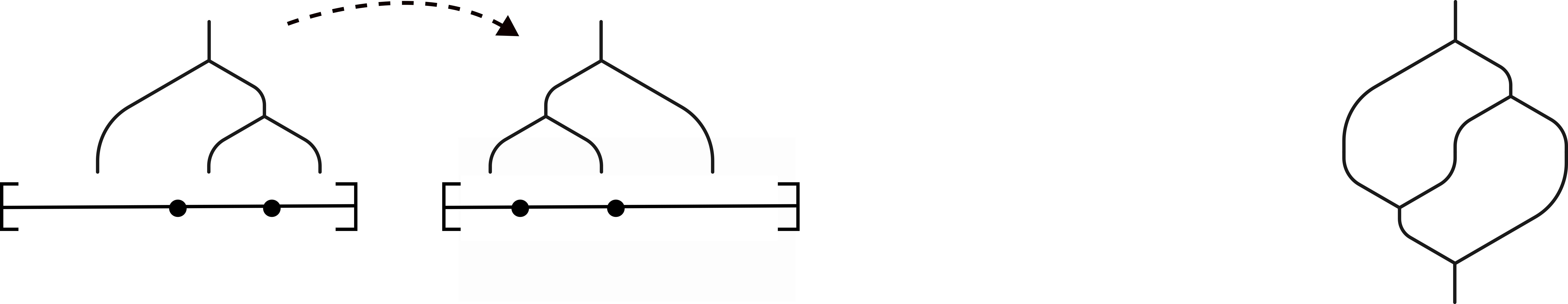}    \put(-325,85){$g(I)=J$} \put(-400,1){$I=\{0, \frac{1}{2}], [\frac{1}{2},\frac{3}{4}],[\frac{3}{4},1]\}$, $J=\{[0,\frac{1}{4}],[\frac{1}{4},\frac{1}{2}],[\frac{1}{2},1]\}$}
  \caption{An element $g \in F$ given by the ordered pair $(I,J)$ and its associated pair of trees.}\label{fig: thompson group trees} 
\end{center}
\end{figure}
Conversely, for every $g \in F$ there is a standard dyadic partition $I$ such that $g(I)$ is standard dyadic. The pair $(I,g(I))$ therefore determines $g$, but this pair is not unique. For any refinement $I'$ of $I$ that is also standard dyadic, $(I',g(I'))$ also represents $g$. The pair of trees given by $(I,g(I))$ will differ from the pair of trees given by $(I',g(I'))$ by finitely many \textit{cancelling carets}. An example of one cancelling caret is shown in Figure \ref{fig: cancelling caret example}.
\begin{figure}[ht]
\begin{center}
\includegraphics[scale=0.9]{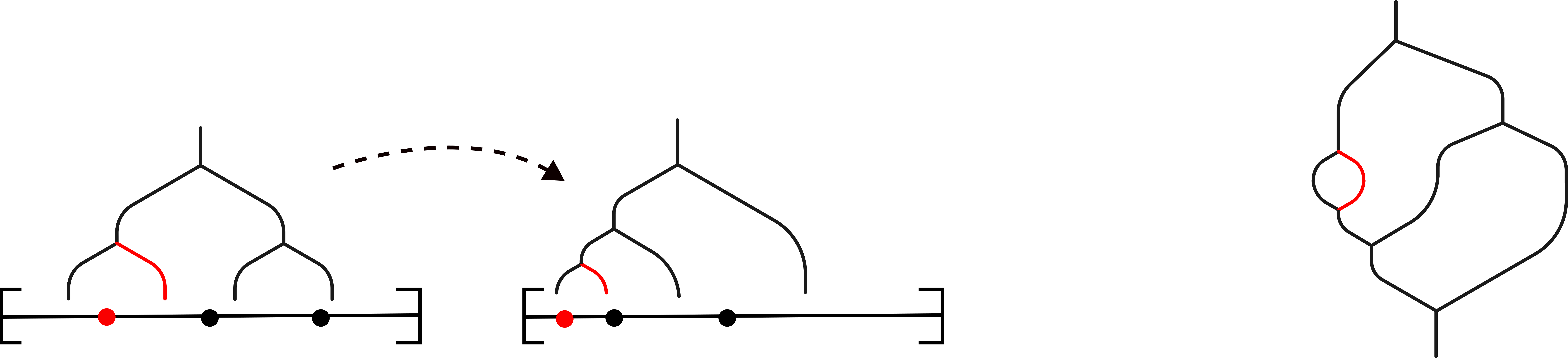} \put(-285,55){$g(I')=J'$} \put(-375,-15){$I'=\{[0, \frac{1}{4}],[\frac{1}{4}, \frac{1}{2}] ,[\frac{1}{2},\frac{3}{4}],[\frac{3}{4},1]\}$, $J'=\{[0,\frac{1}{8}],[\frac{1}{8},\frac{1}{4}],[\frac{1}{4},\frac{1}{2}],[\frac{1}{2},1]\}$.}
  \caption{The same element $g$ as in Figure \ref{fig: thompson group trees}, but the pair of partitions $(I',J')$ is a refinement of $(I,J)$ and the pair of trees differs by a cancelling caret, shown in red.}\label{fig: cancelling caret example}
\end{center}
\end{figure}

We say a pair of trees is \textit{reduced} if it does not have any cancelling carets. For example, the pair in Figure \ref{fig: thompson group trees} is reduced. It is known that elements $F$ are in bijection with reduced pairs of planar, rooted, binary trees \cite{canon}.

Jones introduced a construction of links from the elements of Thompson group $F$,  using their associated pairs of trees \cite{jones14}. In fact, it was shown in \cite{jones14} that this method gives rise to every link type.

We will generalize Jones' construction to build $(n,n)$-tangles from $F$ using \textit{strand diagrams}, objects first introduced in \cite{belk_thesis,belk_matucci}.
These can be thought of as generalizations of the previously discussed pair-of-trees diagrams. We define them as in \cite{belk_thesis}: 

\begin{defn} \label{def: strand diagram} An $(m,n)$-strand diagram $\Gamma$ is a finite graph embedded in $[0,1] \times [0,1]$ such that:
\begin{itemize}
    \item $\Gamma$ has $m$ univalent vertices  on the top edge of the square, and $n$ univalent vertices on the bottom edge of the square.
    \item Every vertex in the interior is either a merge or a split (see Figure \ref{fig: merge and split})
    \item All edges have nonzero slope.
\end{itemize}
\end{defn}

\begin{figure}[ht]
\begin{center}
\includegraphics{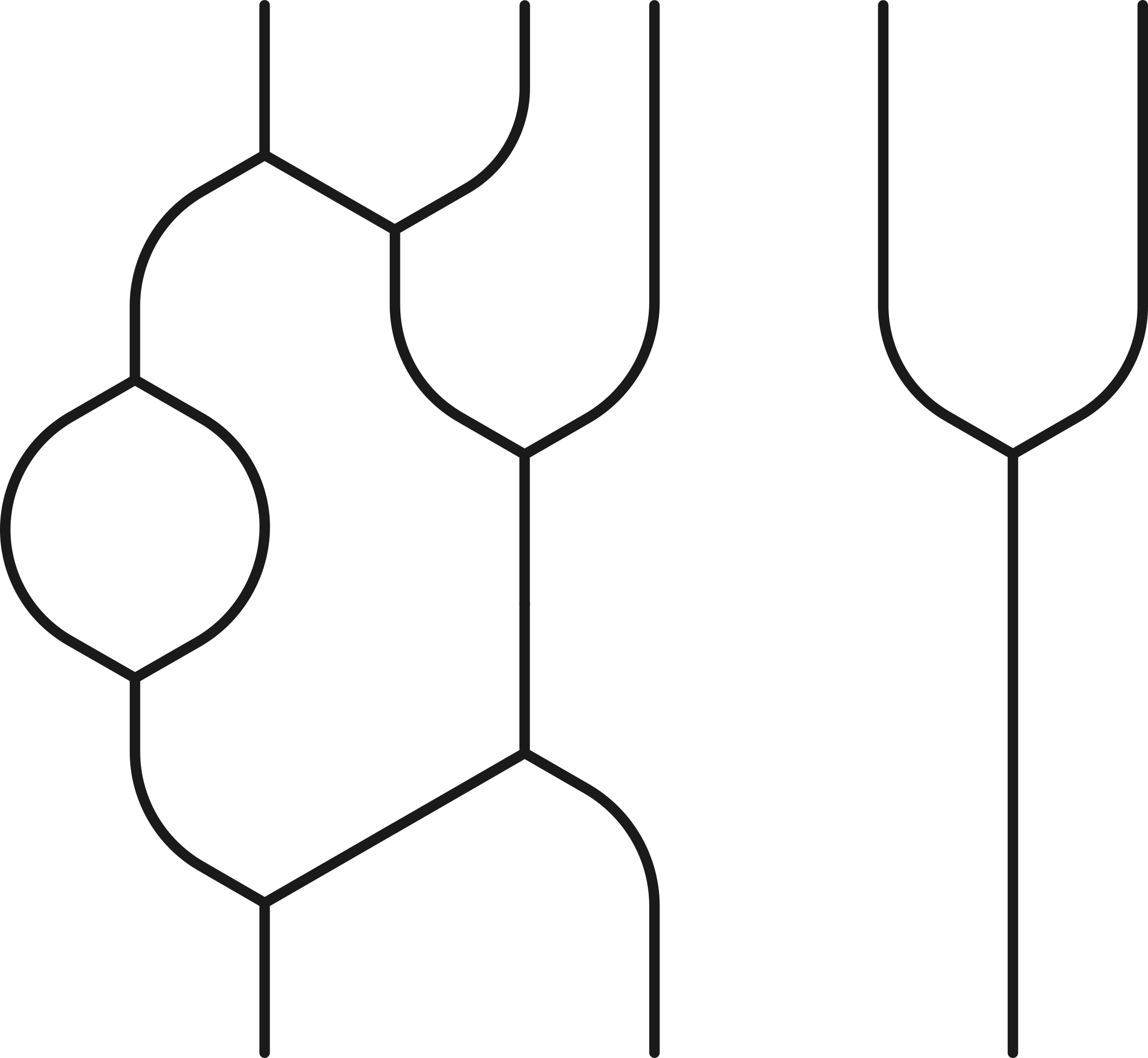}
\caption{an example of a $(5,3)$-strand diagram}\label{strand diagram example}
\end{center}
\end{figure}

Strand diagrams are similar to braids, but instead of twists there are merges and splits, which may cause the number of points at the bottom of the square to differ from the number of points at the top. 
\begin{figure}[ht]
\begin{center}
\includegraphics{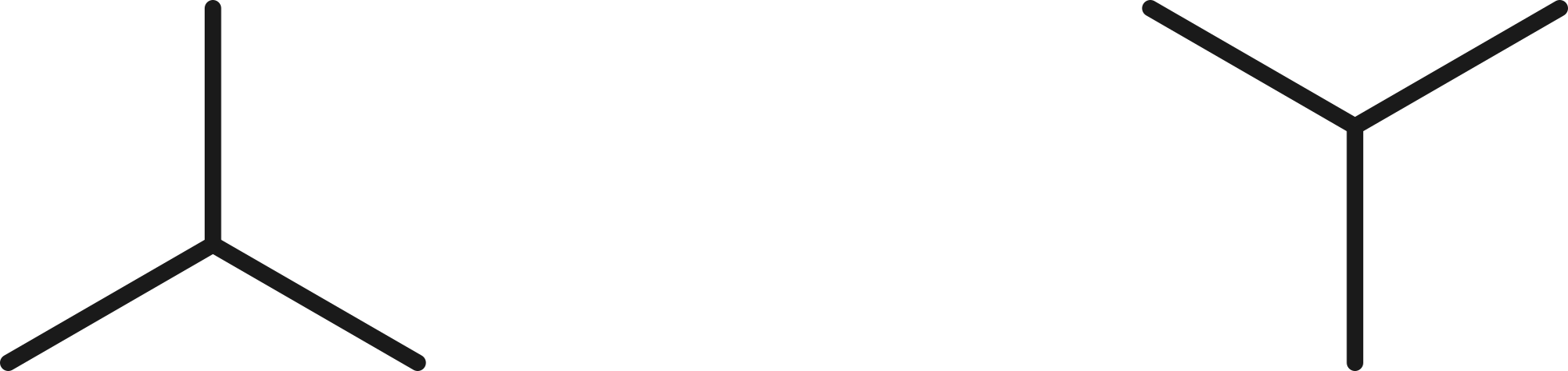}
\caption{A split (left) and a merge (right).}\label{fig: merge and split}
\end{center}
\end{figure}
Isotopic strand diagrams are considered to be equal. We denote the collection of $(m,n)-$strand diagrams as $\mathcal{D}^{m}_{n}$. 

Strand diagrams can be reduced using two local moves, pictured in Figure \ref{fig: moves}. Cancelling carets can be seen as Type I moves in $\mathcal{D}^{1}_{1}$.
\begin{figure}[ht]
\begin{center}
\includegraphics{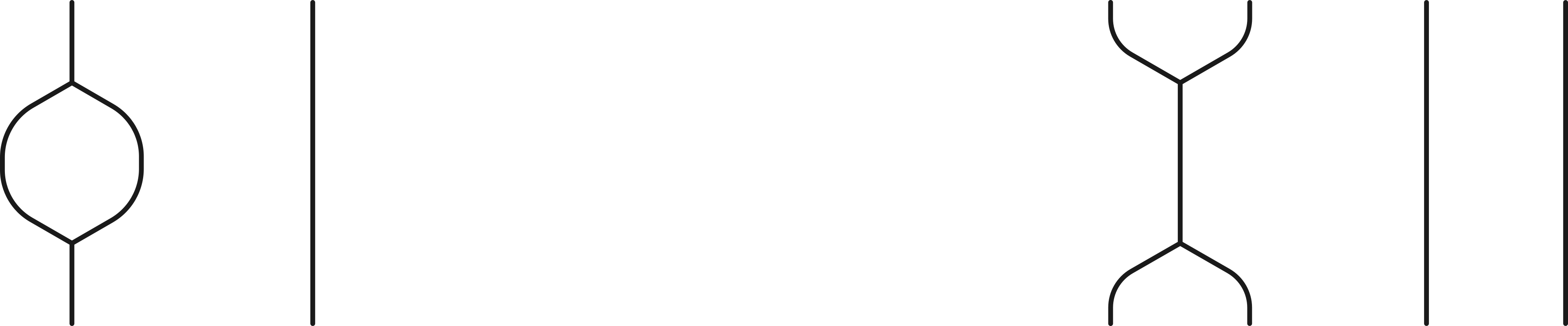}\put(-248,25){$\longrightarrow$}\put(-52,25){$\longrightarrow$}\caption{Cancellation moves of type I (left) and Type II (right). Type I moves are applied when a split contains a merge directly below it, and the two share a pair of edges. Type II moves are applied when a merge occurs directly above a split, and the two share one edge.}\label{fig: moves}
\end{center}
\end{figure}

Strand diagrams related by a finite sequence of these moves and their inverses are called equivalent. We let $\overline{\mathcal{D}^{m}_{n}}$ refer to equivalence classes of $(m,n)-$strand diagrams.

We call a strand diagram \textit{reduced} if it is not subject to any further reductions, and we denote by $\mathcal{R}^{m}_{n}$ the set of reduced $(m,n)$-strand diagrams. Every $(m,n)$-strand diagram is equivalent to a unique reduced $(m,n)$-strand diagram, therefore, the reduction map $\rho\colon \overline{\mathcal{D}^{m}_{n}} \to \mathcal{R}^{m}_{n}$ is a bijection.

Given $\Gamma_{1} \in \mathcal{D}^{k}_{m}$ and $\Gamma_{2} \in \mathcal{D}^{m}_{n}$, we denote $\Gamma_{2} \circ \Gamma_{1} \in \mathcal{D}^{k}_{n}$ to be the result of placing $\Gamma_{2}$ below $\Gamma_{1}$ and joining them at their common endpoints. This vertical stacking induces a well-defined binary operation on $\overline{\mathcal{D}^n_n}$, from which we get a group structure. The identity element in $\overline{\mathcal{D}^n_n}$ is $[{\rm Id}_n]$, the class of the trivial $(n, n)$-strand diagram, which is $n$ vertical strands. The inverse of any $[\Gamma] \in \overline{\mathcal{D}^n_n}$ is $[\Gamma^*]$, where $\Gamma^*$ is the vertical reflection of $\Gamma$. Given $\Gamma_{1} \in \mathcal{R}^{k}_{m}, \Gamma_{2} \in \mathcal{R}^{m}_{n}$ we let $\Gamma_{2} * \Gamma_{1} \in \mathcal{R}^{k}_{n}$ denote the result of vertically stacking and then reducing.
Observe that $\mathcal{R}^{n}_{n}$ is a group with multiplication $*$. It can also be considered as a subset of $\mathcal{D}^{n}_{n}$, but it is not a subgroup as $\mathcal{R}^{n}_{n}$ is not closed with respect to the $\circ$ operation.

 Recall that elements of Thompson's group $F$ are in bijection with reduced pairs of planar rooted binary trees, which can now be recognized as elements of $\mathcal{R}^{1}_{1}$. 
 
 \begin{defn}\label{def:delta}Let $\delta\colon F \to \mathcal{R}^{1}_{1}$ be the map sending a function in the Thompson group to its reduced pair of trees. By \cite[Prop $2.5$]{belk_matucci}, $\delta$ is a group isomorphism.\end{defn}

More generally, $F \cong \mathcal{R}^{n}_{n}$ for any $n \in \mathbb{N}$, via isomorphisms $ \mathcal{R}^1_1 \to \mathcal{R}^n_n$ given in \cite{matucci}. For a $(1, 1)$-reduced strand diagram $\Gamma$, this isomorphism sends $\Gamma$ to $v_n * \Gamma * v_n^*$, where $v_n$ is the ``right vine" with $n$ leaves (see Figure \ref{fig: right vine}). This map represents conjugation by $v_{n}$, and has an inverse map which is conjugation by $v_{n}^{*}.$ More generally, this conjugation map gives an isomoprhism if we replace $v_{n}$ with any element of $\mathcal{R}^{1}_{n}.$ Specifically, this paper will focus on isomorphisms that use symmetric trees $S_{k} \in \mathcal{R}^{1}_{2^{k}}$ (see Figure \ref{fig: S3}).

\begin{figure}[ht]
\centering
\includegraphics[scale=2]{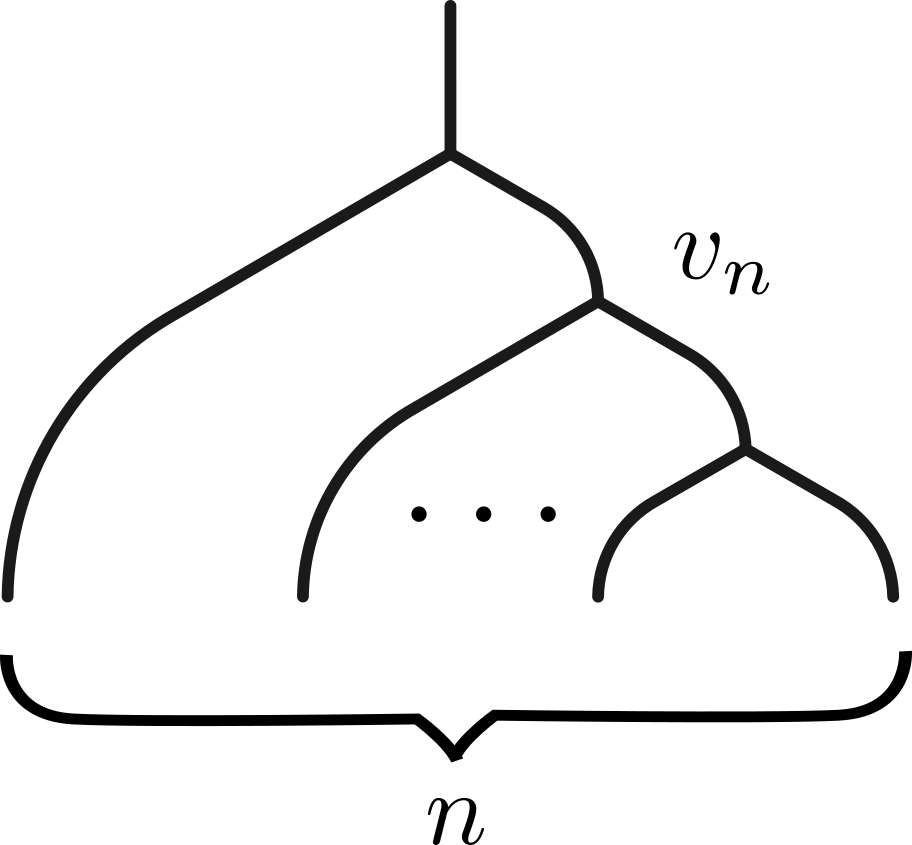}
\caption{The right vine $v_{n}$}\label{fig: right vine}
\end{figure}

\begin{defn}\label{def: gammakg}
Let $\Theta_{k}\colon F\to  \mathcal{R}^{2^{k}}_{2^{k}}$ be the isomorphism resulting from conjugating by the symmetric tree with $2^k$ leaves, i.e. $\Theta_{k}(g)=S_{k}*\delta(g)*S_{k}^{*}.$  Observe $\Theta_{0}=\delta$ as in Definition \ref{def:delta}
\end{defn}

\begin{figure}[ht]
\begin{center}
    \includegraphics[scale=1.1]{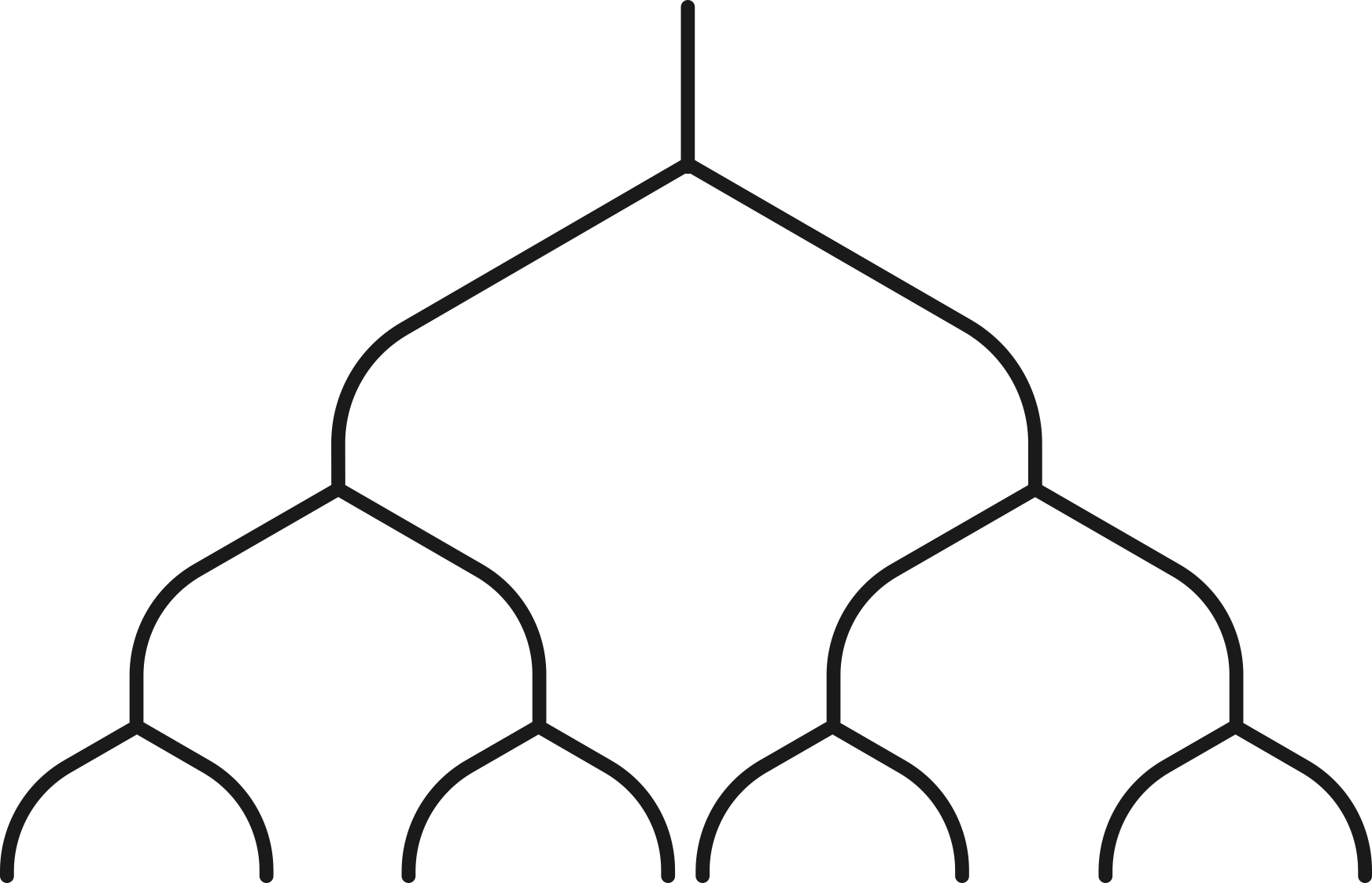}
    \caption{A symmetric tree $S_{3}$, with $8$ leaves}\label{fig: S3}
\end{center}
\end{figure}
\section{From elements of Thompson's groups $F$ to tangles}\label{sec: tangles}

The aim of this section is to introduce the method by which we construct $(2^{k+1},2^{k+1})$-tangles from elements $g \in F$, and to prove that this construction is asymptotically faithful. Specifically, we will define a functor $T$ sending strand diagrams to unoriented tangles, as a generalization of Jones' functor from forests to tangles introduced in \cite{jones14}. A forest is a strand diagram with only splits. We follow the convention of Jones in \cite{jones14} that forests grow downward, see Figure \ref{fig: forest ex}. Forests growing upward will be referred to as \textit{inverse forests}. 
By specifying a functor from strand diagrams to tangles, each $g \in F$ will have an associated $(2^{k+1},2^{k+1})$ tangle $T(\Theta_{k}(g))$, where $\Theta_{k}$ was given in Definition \ref{def: gammakg}. 

\begin{center}
\begin{figure}[ht]
\includegraphics[scale=1.2]{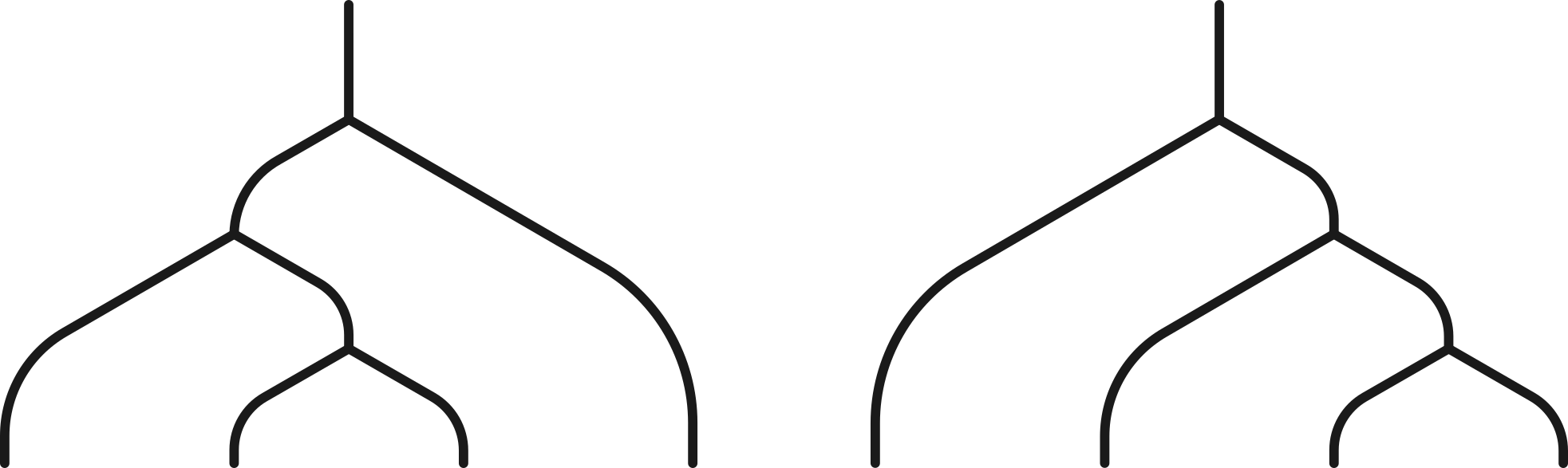}\caption{An example of a forest $D \in \mathcal{D}^{2}_{8}$.}\label{fig: forest ex}
\end{figure}
\end{center}

\begin{defn}
    Let $\mathfrak{D}$ be the category whose objects are positive integers, and the morphisms from $m$ to $n$ are $(m, n)$-strand diagrams. Compositions are concatenations of strand diagrams.
	
    Similarly, let $\mathfrak{T}$ be the category whose objects are positive integers, and the morphisms from $m$ to $n$ are unoriented $(2m,2n)$-tangles. Compositions are concatenations of tangles.
\end{defn}

Now we define a functor $T\colon \mathfrak{D} \to \mathfrak{T}$. First, $T$ acts by the identity on objects. Given an $(m, n)$-strand diagram $\Gamma$ embedded in $[0, 1] \times [0, 1]$, we construct an $(2m, 2n)$-tangle $T(\Gamma)$ as follows.

Observe that the complement of $\Gamma$ is a disjoint union of some polygons $P_1, P_2,..., P_N$. Assume $P_i$ is neither the leftmost region nor the rightmost region, then the boundary of $P_i$ consists of:

\begin{enumerate}
    \item only some edges of $\Gamma$, or
    \item some edges of $\Gamma$ and an interval on the lower side, or
    \item some edges of $\Gamma$ and an interval on the upper side, or
    \item some edges of $\Gamma$, an interval on the lower side, and an inteval on the upper side
\end{enumerate}

See Figure \ref{fig: strand diagram to tangle example} where the regions are labeled 1 -- 4, according to the definition above.

\begin{prop}
    If $P_i$ has boundary of type (1), then $P_i$ has a unique maximum point at a split, and a unique minimum point at a merge. If $P_i$ has boundary of type (2), then $P_i$ has a unique maximum point at a split. If $P_i$ has boundary of type (3), then $P_i$ has a unique minimum point at a merge.
\end{prop}

\begin{proof}
We give a proof for the regions of type (1); the argument for the other two cases is directly analogous. The given polygonal region $P_i$ is topologically a disk. A priori, there are two types of minima and maxima of the {\em boundary} $\partial D$ of a planar region $D$. This first type is where the minimum of the boundary is also a minimum of the region $D$, and similarly a maximum of $\partial D$ is also a maximum of $D$. The second type is where a minimum, respectively maximum of $\partial D$ is a maximum, respectively minimum of $D$. Since the edges of a strand diagram have non-zero slopes (see Definition \ref{def: strand diagram}), the boundary of $P_i$ does not have extrema of the second type. The extrema of the first kind take place precisely at merges and splits, as shown in Figure \ref{fig: merge and split region}.
    \begin{figure}[ht]
\begin{center}
\includegraphics{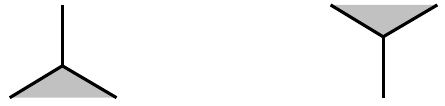}
\scriptsize
\put(-189,6){$P_i$}
\put(-34,40){$P_i$}
\caption{A maximum and a minimum of $P_i$.}\label{fig: merge and split region}
\end{center}
\end{figure}
Suppose $P_i$ has more than one local minimum.
It follows from ambient Morse theory (or in this case, also from planar topology) that, looking at sublevel sets of the height function, the minima of the first kind give rise to $0$-handles of $P_i$. Since $P_i$ is connected, these $0$-handles eventually must be connected by $1$-handles. Attaching a $1$-handle to $P_i$ correspond to a maximum of $\partial P_i$ of the second type. Since this is impossible in a strand diagram, $P_i$ has a unique minimum. The same argument shows that there is also a unique maximum. 
\end{proof}

Consider each polygon $P_i$ which is not leftmost or  rightmost. If $P_i$ has boundary of type (1), we introduce an edge connecting the unique maximum point and the unique minimum point of $P_i$. If $P_i$ has boundary of type (2), we  add an edge connecting the unique maximum point of $P_i$ to a point in the interior of the interval in $\partial P_i$ on the bottom of the square. The definition for type (3) is analogous. If $P_i$ has boundary of type (4), we connect a point in the interior of the top interval to a point in the interior on the bottom. Let $\Gamma'$ denote the resulting graph. The new edges are drawn red in Figure \ref{fig: 4-valent vertices to crossings}.

Note that $\Gamma'$ is a graph with $(2m-1)$ univalent vertices on the top side of the square, $(2n - 1)$ univalent vertices on the bottom side, and all other vertices are $4$-valent, containing either a split or a merge of the original strand diagram. We replace each 4-valent vertex by a crossing as in Figure \ref{fig: 4-valent vertices to crossings}. Finally, we add a trivial strand on the leftmost side to get an unoriented $(2m, 2n)$-tangle $T(\Gamma)$. See Figure \ref{fig: strand diagram to tangle example} for an example.

\begin{figure}[ht]
    \centering
    \includegraphics{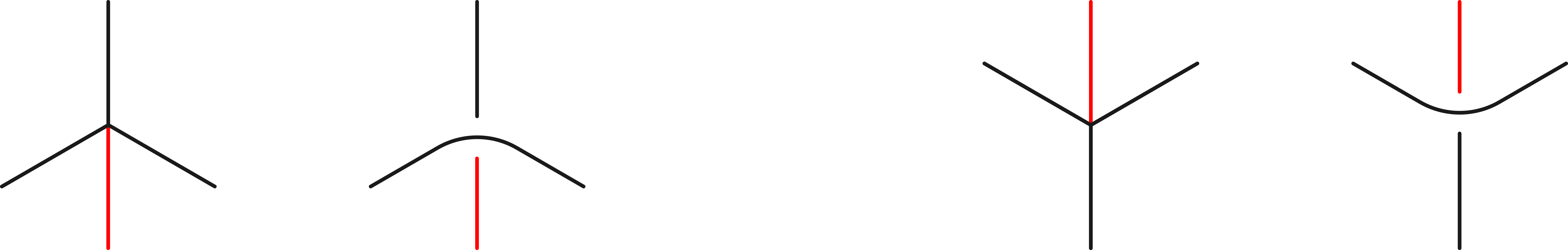}
    \put(-310,30){$\longrightarrow$}
    \put(-370,40){split}
    \put(-75,20){$\longrightarrow$}
    \put(-155,15){merge}
    \caption{The local replacement at $4$-valent vertices of $\Gamma'$, with added edges in $\Gamma'$ marked in red.}
    \label{fig: 4-valent vertices to crossings}
\end{figure}

\begin{figure}[ht]
    \centering
    \includegraphics{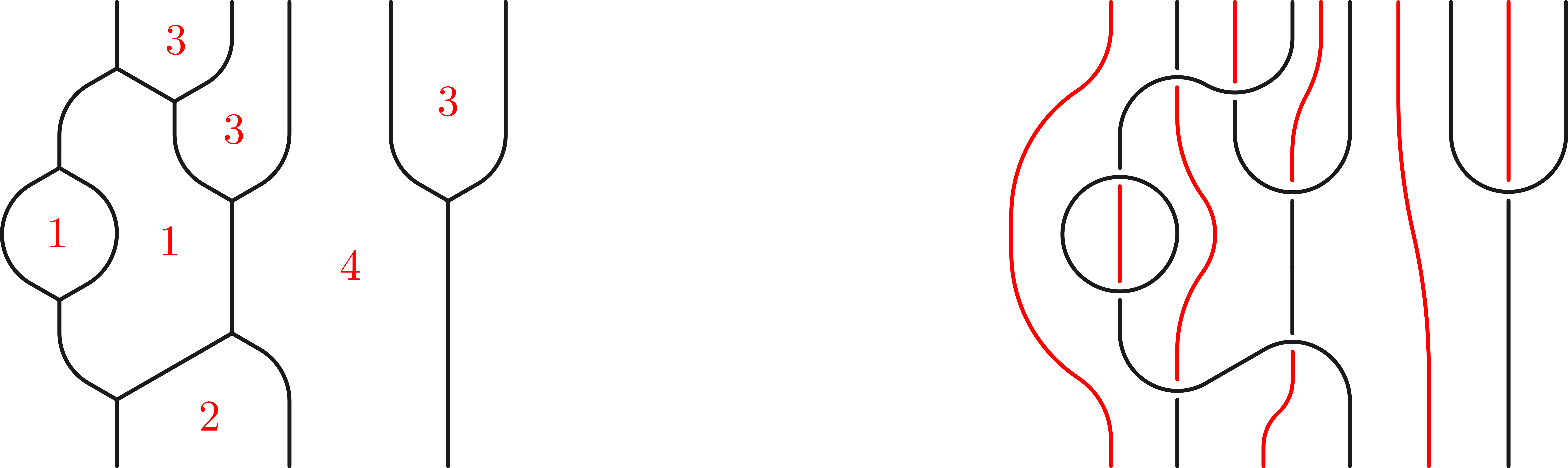}
    \put(-280,-20){$\Gamma$}
    \put(-70, -20){$T(\Gamma)$}
    \caption{A $(5, 3)$-strand diagram $\Gamma$ and its associated $(10, 6)$-tangle $T(\Gamma)$. Each polygon $P$ in the complement of $\Gamma$ (except the leftmost one and the rightmost one) is marked by a number $k$, meaning that $P$ has boundary of type $(k)$.}
    \label{fig: strand diagram to tangle example}
\end{figure}

If $E$ is a forest, then the complement of $E$ does not contain any polygons with boundary of type (1) and (3). In the definition of $T(E)$ above, we just need to connect each split with the interval on the bottom of the square, and we add a strand connecting the top and bottom intervals of each type (4) polygon except the rightmost one. Each split in $E$ gives rise  to an arc connecting two points on the bottom edge of the square in $T(E)$, and we refer to these arcs as \textit{turnbacks}. Jones also constructed tangles from forests in \cite{jones14}; comparing $T(E)$ with Jones' functor, the only difference is that in our construction, $T(E)$ has an extra trivial strand on the left. For an example, see Figure \ref{fig: forest tangle ex}. 
\begin{center}
\begin{figure}[ht]
\includegraphics[scale=1.2]{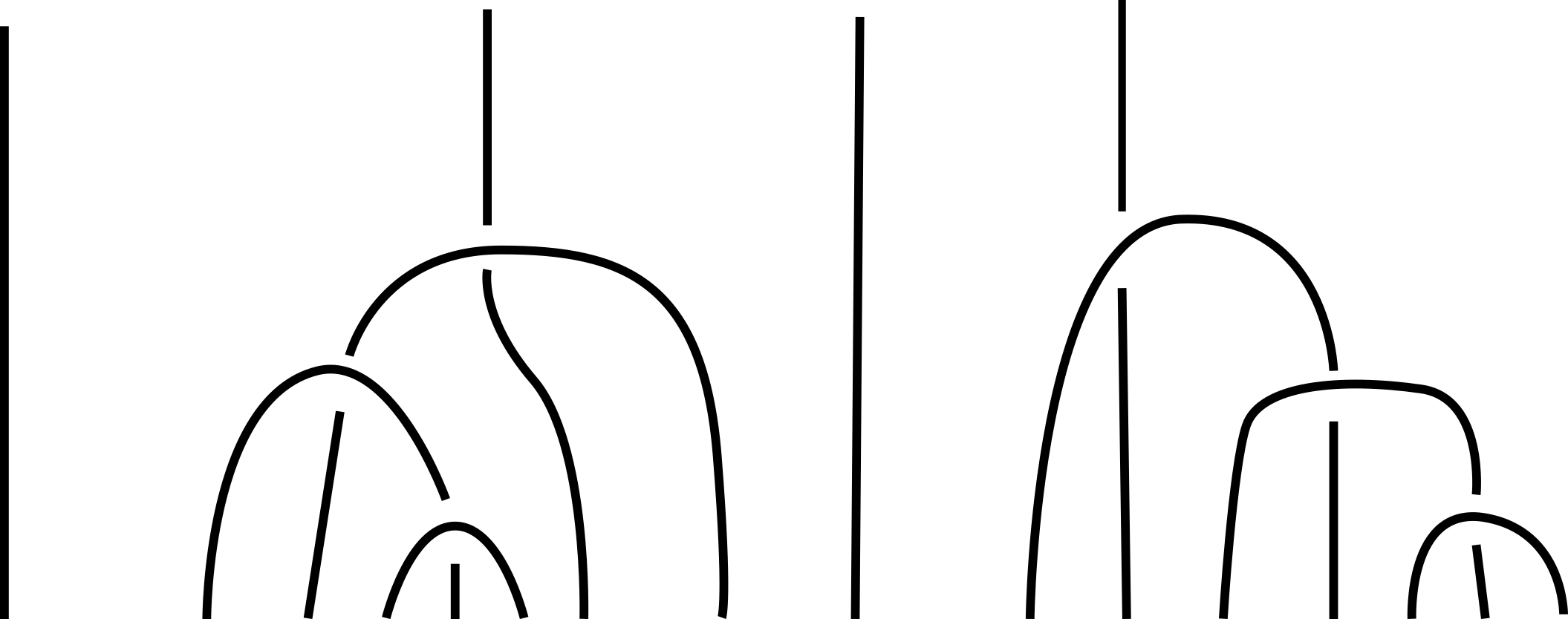}\caption{The tangle $T(D)$, with $D$ as in Figure \ref{fig: forest ex}.}\label{fig: forest tangle ex}
\end{figure}
\end{center}

\begin{defn}\label{def: biforest}
    A strand diagram $\Gamma$ is called a {\em biforest} if $\Gamma = F_1 \circ F_2^{*}$ for some forests $F_1$ and $F_2$.
\end{defn} 

\begin{figure}[ht]
    \centering
    \includegraphics{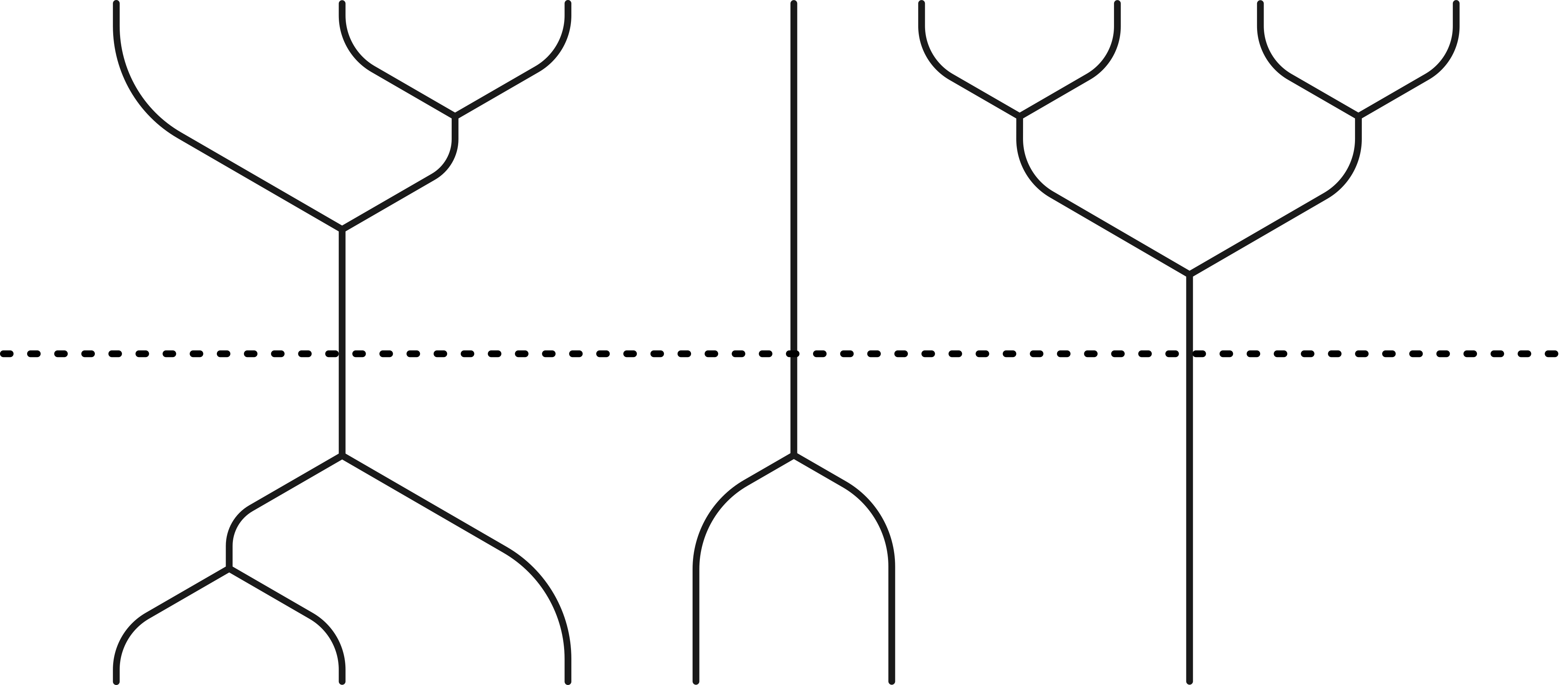}
    \put(-210,60){$F_{2}^{*}$}
    \put(-210,20){$F_{1}$}
    \caption{A biforest in $\mathcal{D}^{8}_{6}$}
    \label{fig: biforest example}
\end{figure}

\begin{prop}\label{prop: biforest}
    Given two biforests $\Gamma$ and $\Gamma'$ with the same number of endpoints, $\Gamma \neq \Gamma'$ implies that $T(\Gamma) \not\cong T(\Gamma')$. 
\end{prop}

Here tangles are considered equivalent, $T(\Gamma) \cong T(\Gamma')$, if they are related by an isotopy fixing the endpoints.

\begin{proof}
    Let $\Gamma=F_1 \circ F_2^*$, $\Gamma'=G_1\circ G_2^*$. Since $\Gamma\neq \Gamma'$, $F_i\neq G_i$ for some $i=1,2$. Without loss of generality, suppose $F_1\neq G_1$. Then $T(\Gamma)$ and $T(\Gamma')$ differ by at least one turnback near $[0,1]\times \{ 0\}$.
\end{proof}

\begin{rem}Proposition \ref{prop: biforest} does not hold for arbitrary strand diagrams. For example, an infinite family of $(1,1)$-strand diagrams, coming from an infinite family of distinct Thompson group elements given by the ordered pairs $(\{[0,\frac{1}{2}],[\frac{1}{2},\frac{3}{4}],[\frac{3}{4},\frac{7}{8}],\dots\},\{\dots, [\frac{1}{8},\frac{1}{4}],[\frac{1}{4},\frac{1}{2}],[\frac{1}{2},1]\})$, produce the trivial tangle with two vertical strands. Similarly, it is known that Jones' construction in \cite{jones14} produces the unlink from many distinct elements of $F$. 
\end{rem}

\subsection{The statement and the proof of faithfulness} \label{sec: faithful}

In this section we prove asymptotic faithfulness of the construction of tangles from $F$. More precisely, we show that given any $g \neq h \in F$, there exists some $K$ such that for $k \geq K$, $T(\Theta_{k}(g)) \not\simeq T(\Theta_{k}(h))$. This statement highlights a key difference between the creation of tangles from $F$ in this paper, and that of links in \cite{jones14}. In general, the construction of links from $F$ is not faithful. 

In this section we refer to the \textit{maximum height} of a tree $S$ to be the maximum distance from a leaf in $S$ to the root. For example, the maximum height of $S_{k}$ is $k$. 
\begin{lem}\label{lemma: difference}
    Let $\Gamma \in \mathcal{R}_{n}^{1}$, i.e. $\Gamma$ is a binary tree with $n$ leaves. If $d$ is the maximum height of $\Gamma$ and $k\geq d$, then $\Gamma * S_{k}^{*}\in \mathcal{R}^{2^{k}}_{n}$ contains only merges.
\end{lem}
Imprecisely, the resulting strand diagram $\Gamma * S_{k}^{*}\in \mathcal{R}^{2^{k}}_{n}$ may be thought of as encoding the ``difference" between $S_{k}$ and $\Gamma$.
\begin{proof}
    It follows from the fact that $\Gamma$ is a sub-tree of $S_k$ that $S_{k}$ contains every split of $\Gamma$ (and sometimes more). In the concatenated but unreduced picture $\Gamma \circ S_{k}^{*}$, every split in $T$ and its corresponding merge in $S_{k}^{*}$ can be eliminated by a reduction of Type II. The resulting reduced diagram only has merges, given by merges in $S_{k}^{*}$ that do not correspond to splits in~$T$. 
\end{proof}
\begin{figure}[ht]
\centering
\includegraphics{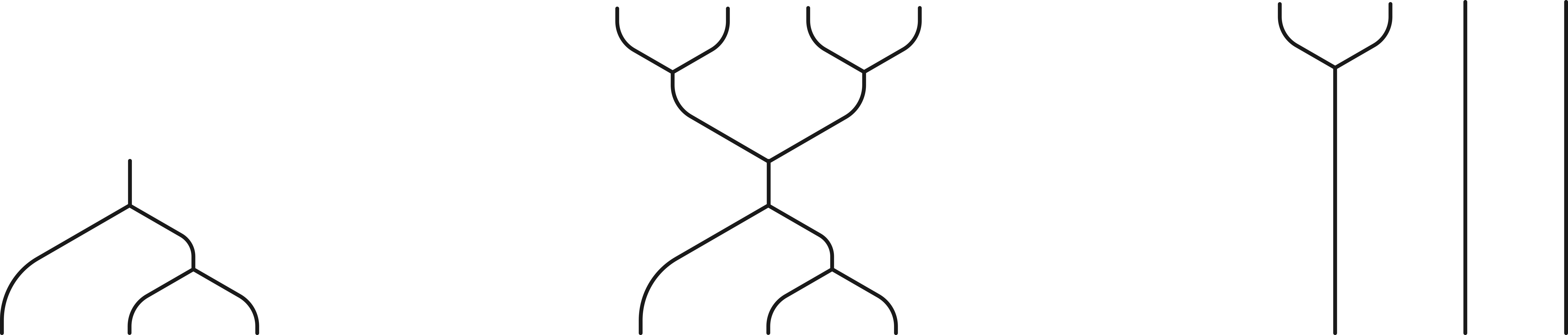}\put(-325, 80){$\Gamma$}\put(-192,80){$\Gamma \circ S_{2}^{*}$}\put(-37,80){$\Gamma*S_{2}^{*}$}
\caption{The tree $S_{2}$ has one more split than the tree $\Gamma$. When $\Gamma \circ S_{2}^{*}$ is reduced to $\Gamma * S_{2}^{*}$, the single merge remaining corresponds to this split.}
\end{figure}

The following result is a more detailed version of Theorem \ref{thm: intro aymptotic faithfulness} in the Introduction. To relate the notation in the two statements, $n=2^k$ and $T_n=T\circ \Theta_k$.

\begin{thm} \label{thm: aymptotic faithfulness for unoriented F}
    Suppose $g, h \in F, g\neq h$. Let $m=\max({\text{height}(g),\text{height}(h)})$. For $k$ such that $2^{k} \geq m$, we have  $T(\Theta_{k}(g))\not\cong T(\Theta_{k}(h))$.
\end{thm}
\begin{proof} Let $\delta(g)=T_{2}^{*}*T_{1}$, and $\delta(h)=U_{2}^{*}*U_{1}$, where $(T_{1},T_{2})$ and $(U_{1},U_{2})$ both are reduced as pairs of trees. Note that \[\Theta_{k}(g)=S_{k}*T_{2}^{*}*T_{1}*S_{k}^{*}=(S_{k} *T_{2}^{*})*(S_{k} * T_{1}^{*})^{*}\] By Lemma \ref{lemma: difference} we know that $S_{k} *T_{2}^{*}$ and $S_{k}*T_{1}^{*}$ both are forests, so the unreduced concatenation $(S_{k} *T_{2}^{*})\circ (S_{k} * T_{1}^{*})^{*}$ is a biforest. However, reducing a biforest results in another biforest, so $\Theta_{k}(g)$ and $\Theta_{k}(h)$ are biforests.

Because $\Theta_{k}$ is an isomorphism, $g\neq h$ implies $\Theta_{k}(g)\neq\Theta_{k}(h)$. Proposition \ref{prop: biforest} implies $T(\Theta_{k}(g))\not\cong T(\Theta_{k}(h))$.
\end{proof}

\section{Orientability}\label{sec: orientability}

The Khovanov chain complexes, considered in Section \ref{sec: Khovanov}, are defined for oriented tangles. To this end,  we will now consider the analogues of the above constructions and results in the oriented case.

\subsection{Oriented strand diagrams and the oriented Thompson group} Orientations of forests were introduced in \cite{HOMFLY}. In this section we extend this to a notion of oriented strand diagrams, from which we can build oriented tangles. 
\begin{defn}
    An orientation on an $(m, n)$-strand diagram $\Gamma$ is an assignment of $+$ or $-$ to each component of the complement of $\Gamma$, except for the rightmost region, such that the signs around each trivalent vertex correspond to one of the four cases shown in Figure \ref{fig: signing rules}. The region on the right of the crossing in each case appears to be unlabelled; however in fact it has a sign determined near its right boundary. An exception is the rightmost region which does not have a sign.
    $\Gamma$ is said to be orientable if it admits an orientation. 
\end{defn}

\begin{figure}[ht]
    \centering
    \begin{subfigure}{0.2\textwidth}
        \centering
        \includegraphics{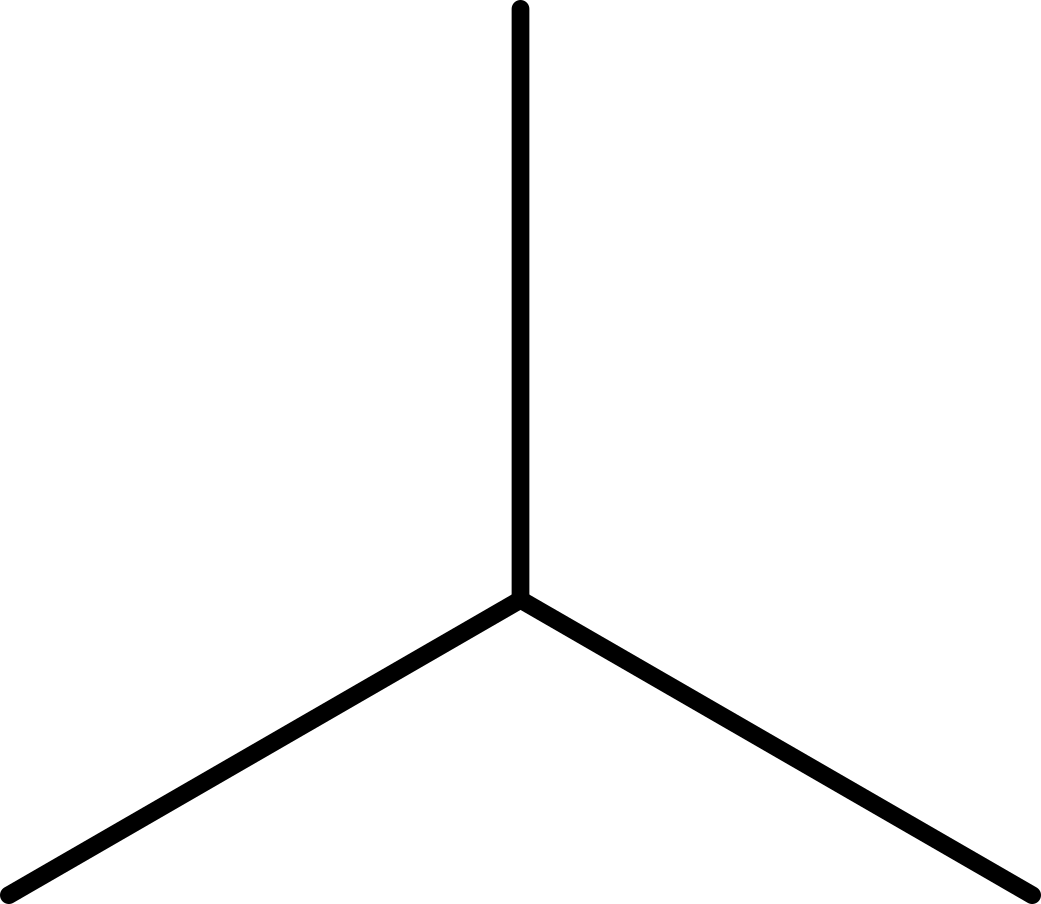}
        \put(-45, 20){$+$}
        \put(-30, -3){$-$}
        \caption{}
        \label{fig: positive split}
    \end{subfigure}
    \begin{subfigure}{0.2\textwidth}
        \centering
		\includegraphics{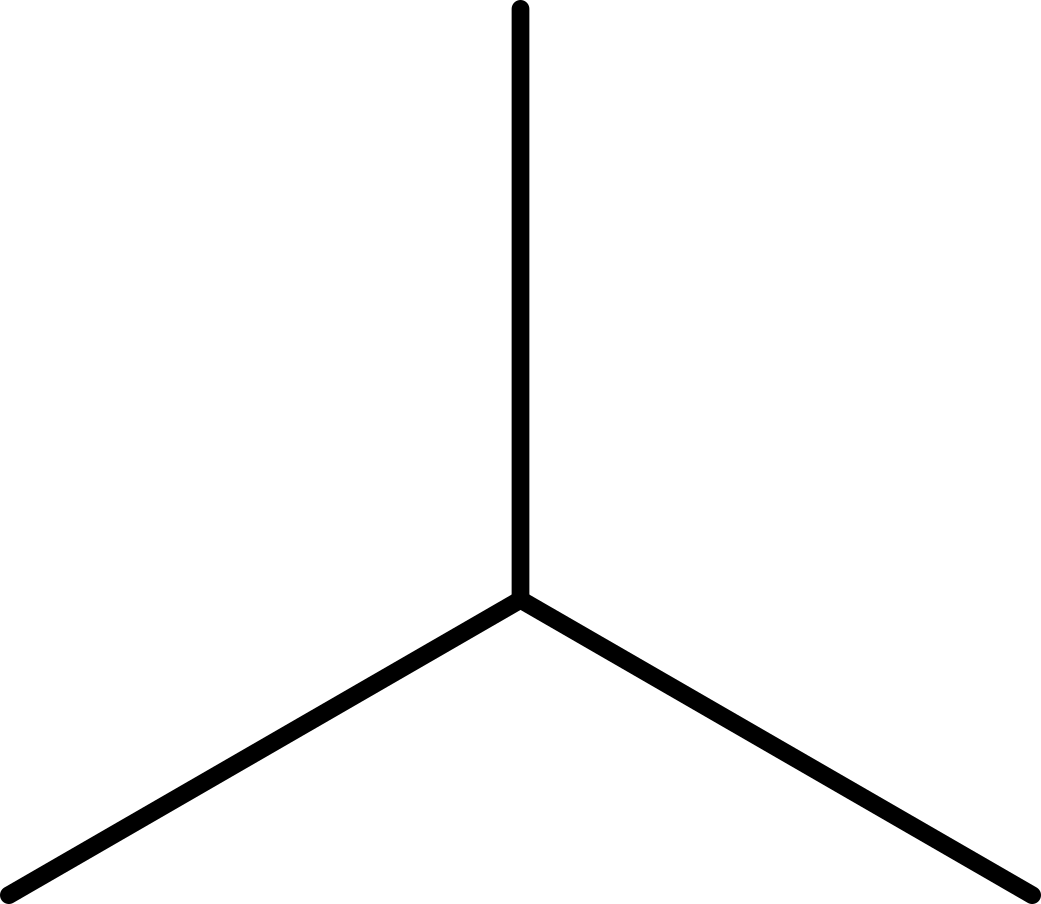}
        \put(-45, 20){$-$}
        \put(-30, -3){$+$}
		\caption{}
		\label{fig: negative split}
    \end{subfigure}
    \begin{subfigure}{0.2\textwidth}
        \centering
		\includegraphics{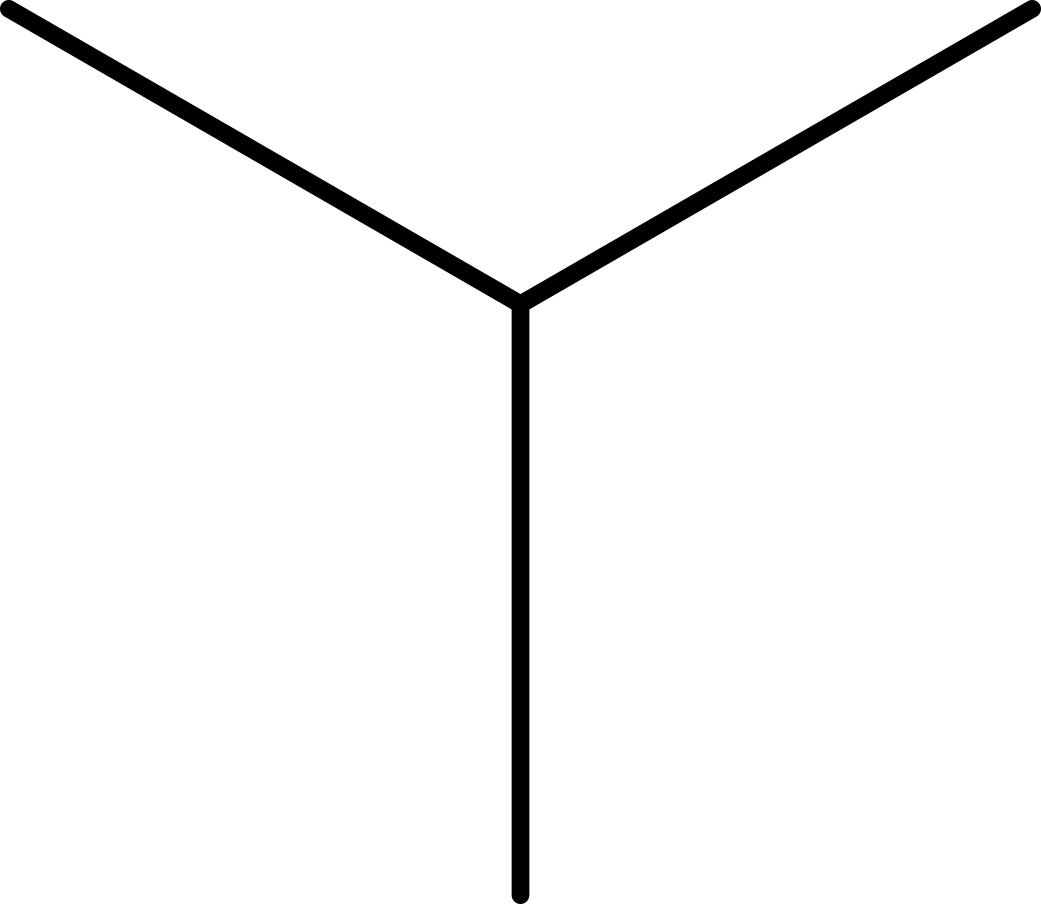}
        \put(-45, 20){$+$}
        \put(-30, 40){$-$}
		\caption{}
		\label{fig: positive merge}
    \end{subfigure}
    \begin{subfigure}{0.2\textwidth}
        \centering
		\includegraphics{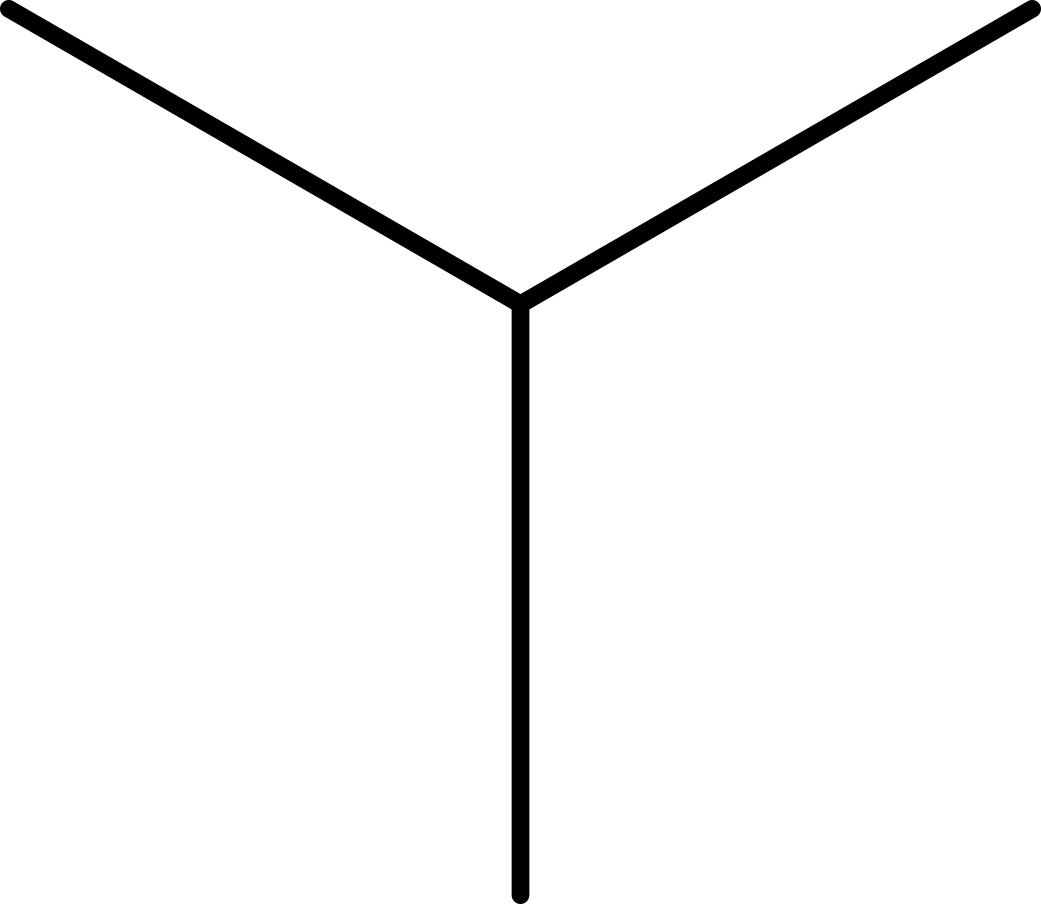}
        \put(-45, 20){$-$}
        \put(-30, 40){$+$}
		\caption{}
		\label{fig: negative merge}
    \end{subfigure}
    \caption{(A) A positive split. (B) A negative split. (C) A positive merge. (D) A negative merge.}
    \label{fig: signing rules}
\end{figure}
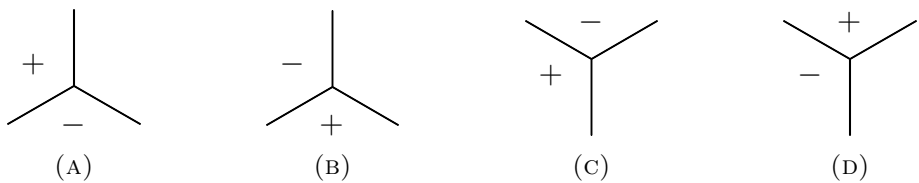
 
Any forest is orientable. We can arbitrarily put a sign in the leftmost region and in the type (4) regions between the components; then the signs in the other regions are determined uniquely (see \cite[Proposition $3.2$]{HOMFLY}). Therefore any forest with $n$ roots admits exactly $2^n$ orientations, and any tree $s$ admits exactly two orientations, corresponding to the two choices of a sign in the leftmost region. We denote the positively oriented $s$ (with $+$ in the leftmost region) by $\vec{s}$. By convention, we will assume that every oriented tree is positively oriented from now on.

\begin{defn} \label{def: n sign}
    An $n$-sign is a sequence of $n$ $+$'s or $-$'s. 
    We will follow the convention in the literature \cite{HOMFLY} that the first two signs in the sequence are $+, -$.
\end{defn}

\begin{defn} \label{def: 2^k n sign}
    If $\nu$ is an $n$-sign, we can replace each $(...,+,...)$ in $\nu$ by $(...,+, -,...)$, and each $(...,-,...)$ by $(...,-, +,...)$, and denote the resulting $2n$-sign by $2\nu$. Let $2^{k}\nu$ denote the $n$-sign that results from repeating this process $k$ times.
\end{defn}

Note that an oriented $(m, n)$-strand diagram $\vec{\Gamma}$ induces a sequence of $m$ signs on the top side of the square, and a sequence of $n$ signs on the bottom side. If both of them start with $+, -$, let $\mu$ denote the top $m$-sign and $\nu$ the bottom $n$-sign $\nu$; in this case we say that $\vec{\Gamma}$ is an oriented strand diagram from $\mu$ to $\nu$, or $\vec{\Gamma}$ is a $(\mu, \nu)$-strand diagram. We denote the collection of $(\mu, \nu)$-strand diagrams by $\mathcal{D}^\mu_\nu$. 

We get an $(m, n)$-strand diagram by forgetting the orientation of a $(\mu, \nu)$-strand diagram.  Also note that an orientation of an orientable $(m, n)$-strand diagram is determined by the $m$-sign or the $n$-sign induced by it. So we may consider $\mathcal{D}^\mu_\nu$ as a subset of $\mathcal{D}^m_n$.

Given a $(\mu, \nu)$-strand diagram $\vec{\Gamma}$, we define $(\vec{\Gamma})^*$ to be the $(\nu, \mu)$-strand diagram obtained by taking vertical symmetry of $\vec{\Gamma}$. Given a $(\mu, \nu)$-strand diagram $\vec{\Gamma}_1$ and a $(\nu, \rho)$-strand diagram $\vec{\Gamma}_2$, we define $\vec{\Gamma}_2 \circ \vec{\Gamma}_1$ to be the $(\mu, \rho)$-strand diagram obtained by concatenating strand diagrams and signed regions. The following extends the notion of a reduced strand diagrams to the oriented case.

\begin{defn} \label{def: oriented reduction moves}
    A reduction of an oriented strand diagram is a sequence of moves shown in Figure \ref{fig: oriented reduction moves}.
	
    An oriented strand diagram is said to be reduced if it is not subject to any reductions. We denote the collection of reduced $(\mu, \nu)$-strand diagrams as $\mathcal{R}^\mu_\nu$; it can be regarded as a subset of $\mathcal{D}^{\mu}_{\nu}$.

    Given a reduced $(\mu, \nu)$-strand diagram $\vec{\Gamma}_1$ and a reduced $(\nu, \rho)$-strand diagram $\vec{\Gamma}_2$, we define $\vec{\Gamma}_2  * \vec{\Gamma}_1$ to be the reduced $(\mu, \rho)$-strand diagram obtained by fully reducing $\vec{\Gamma}_2 \circ \vec{\Gamma}_1$. Similarly to the case of reduced unoriented strand diagrams, $\mathcal{R}^\mu_\mu$ is a group with multiplication~$*$.
\end{defn}

\begin{figure}[ht]
    \centering
    \begin{subfigure}{0.2\textwidth}
        \centering
        \includegraphics{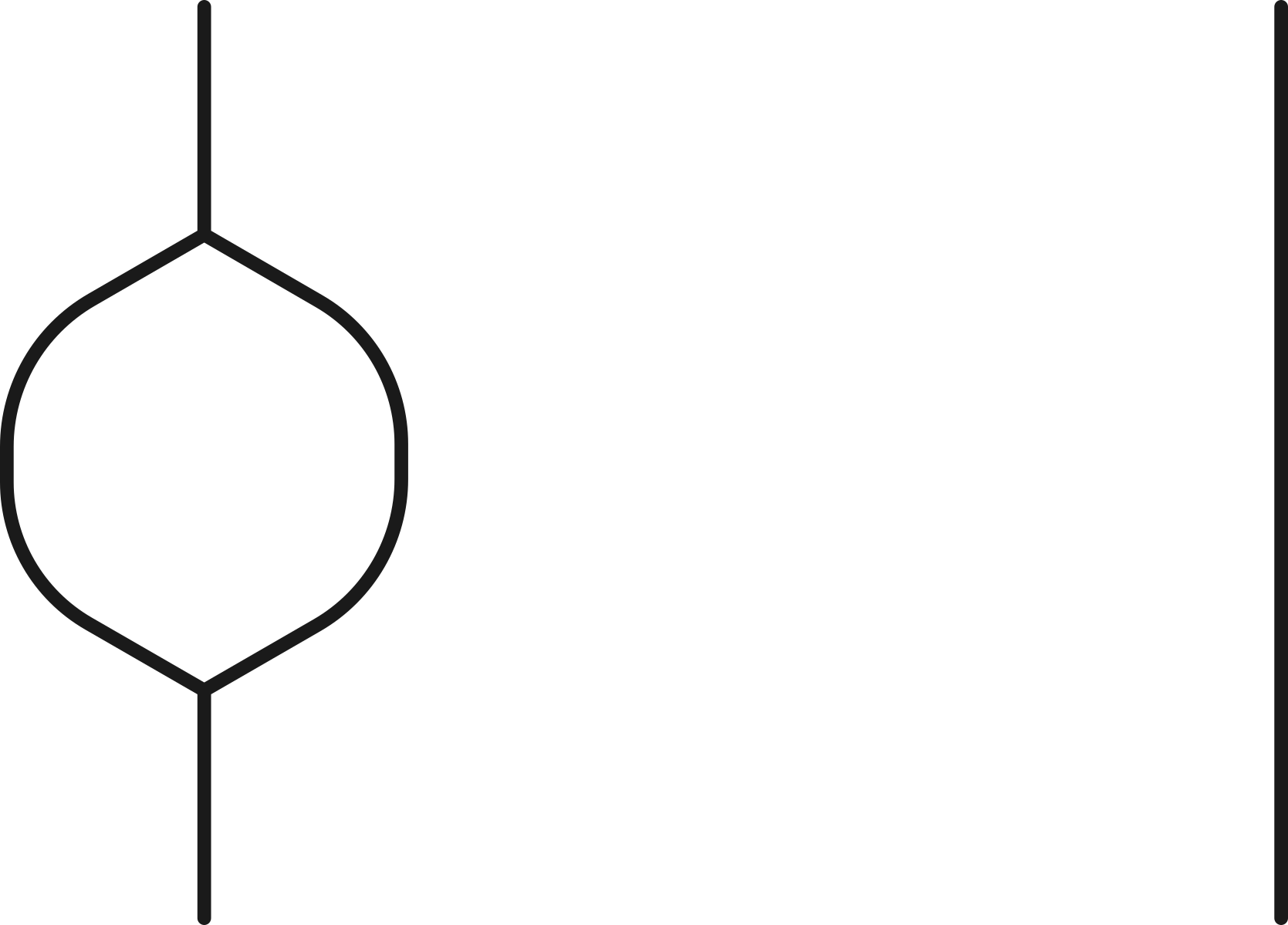}
        \put(-45, 26){$\longrightarrow$}
        \put(-95, 26){\small $+$}
        \put(-72, 26){\small $-$}
        \put(-18, 26){\small $+$}
        \caption{}
    \end{subfigure}
    \hfill
    \begin{subfigure}{0.2\textwidth}  
        \centering 
        \includegraphics{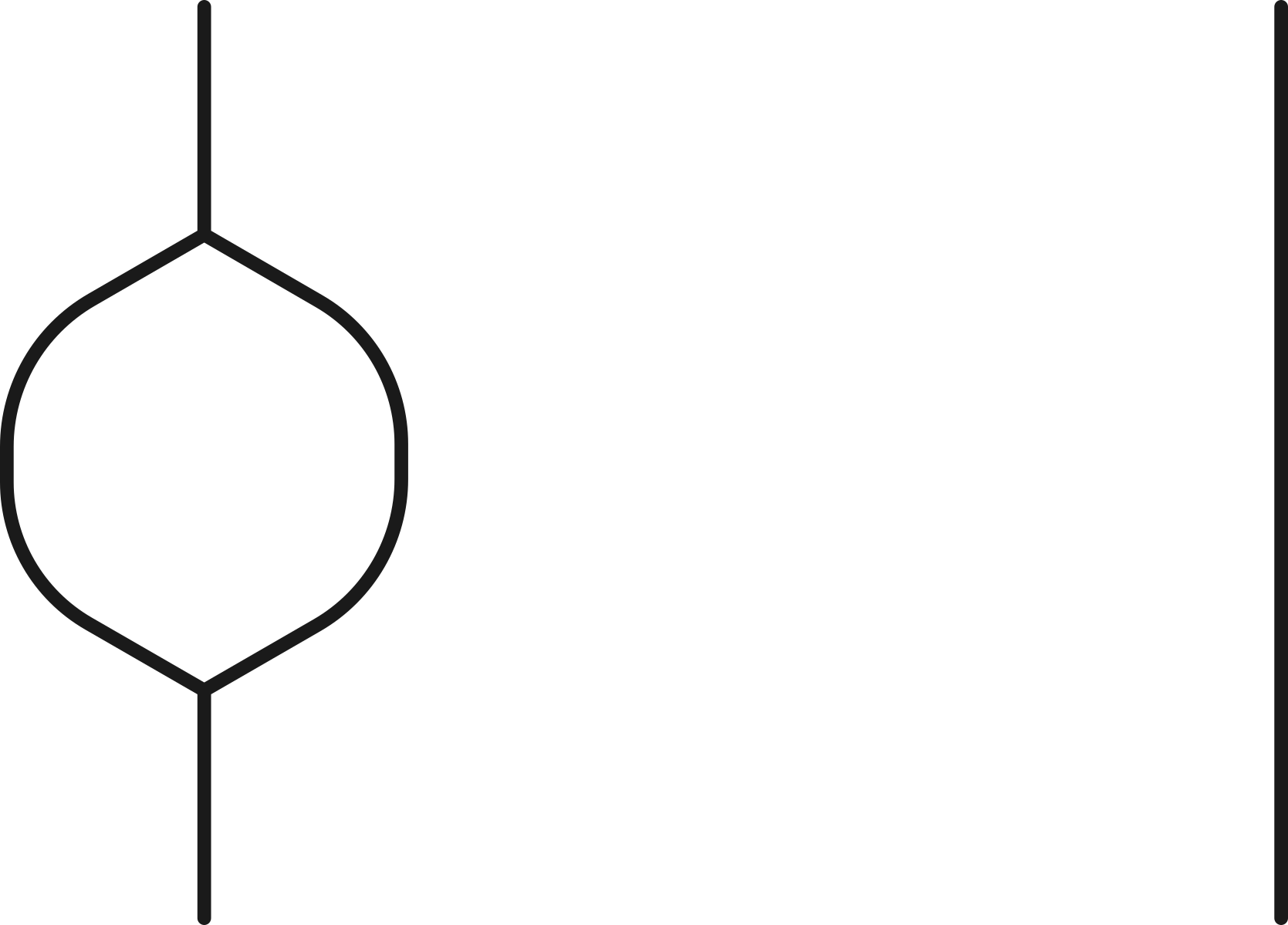}
        \put(-45, 26){$\longrightarrow$}
        \put(-95, 26){\small $-$}
        \put(-72, 26){\small $+$}
        \put(-18, 26){\small $-$}
        \caption{}
    \end{subfigure}
    \hfill
    \begin{subfigure}{0.2\textwidth}   
        \centering 
        \includegraphics{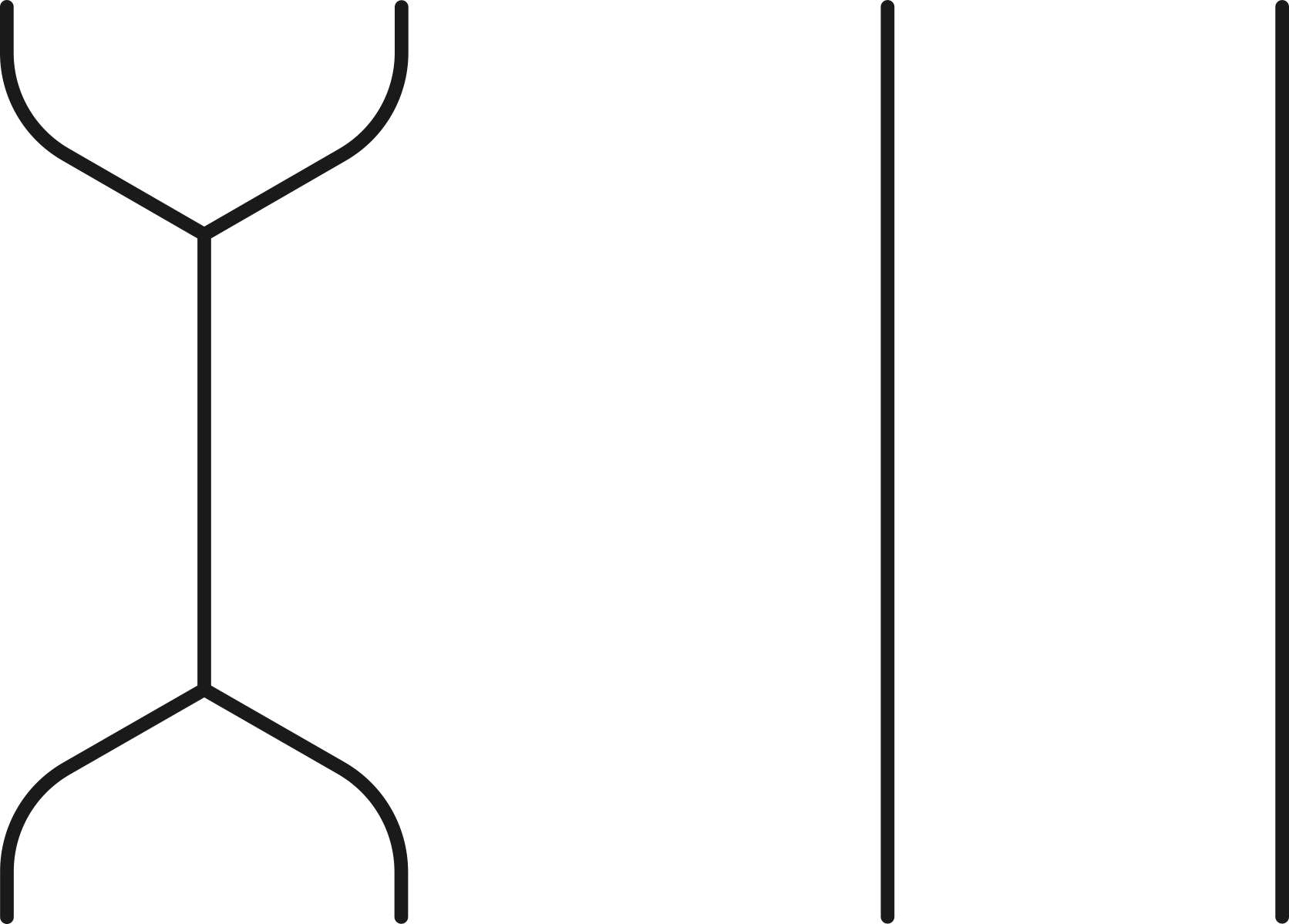}
        \put(-55, 26){$\rightarrow$}
        \put(-85, 26){\small $+$}
        \put(-72, 50){\small $-$}
        \put(-72, 2){\small $-$}
        \put(-38, 26){\small $+$}
        \put(-20, 26){\small $-$}
        \caption{}
    \end{subfigure}
    \hfill
    \begin{subfigure}{0.2\textwidth}   
        \centering 
        \includegraphics{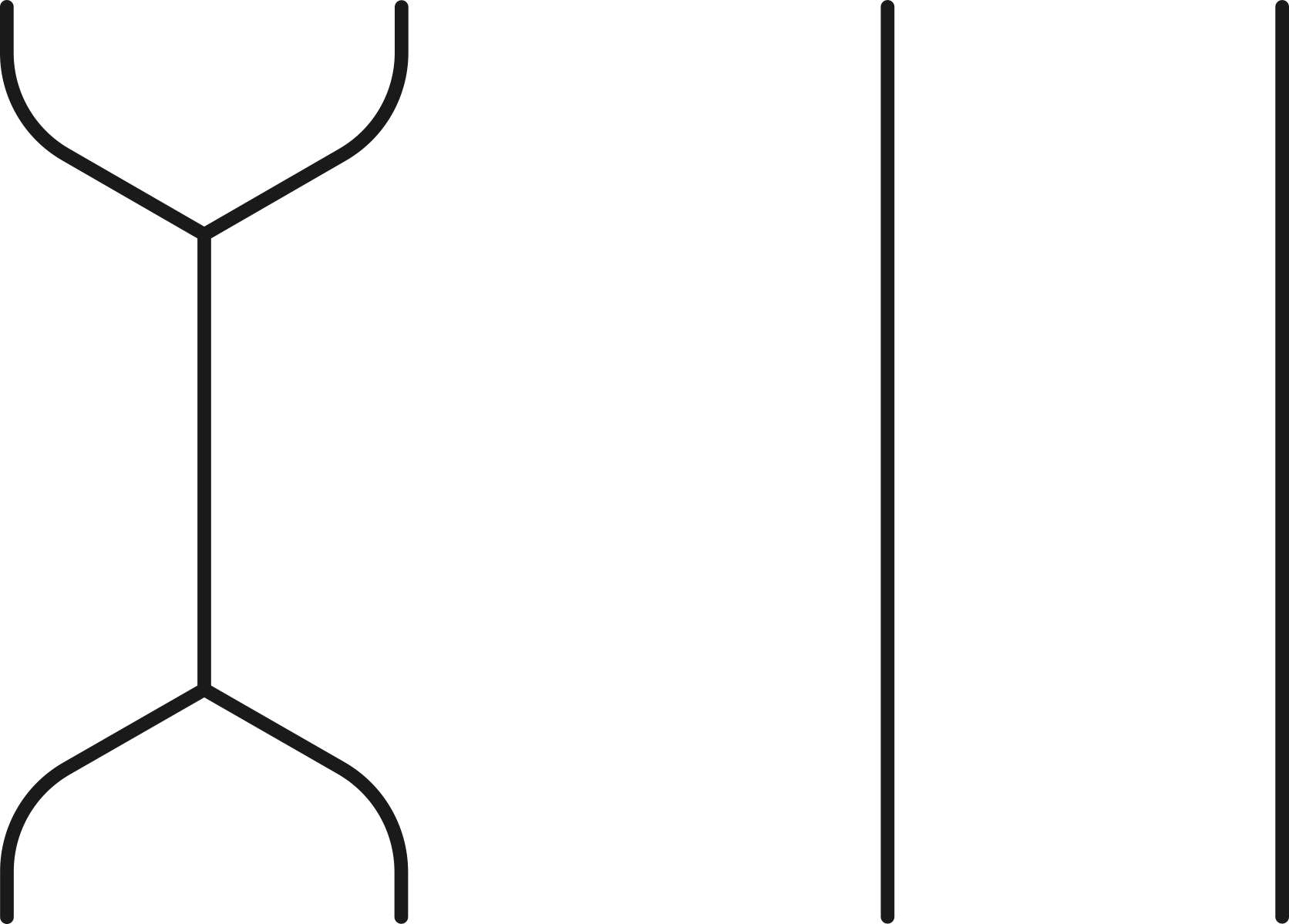}
        \put(-55, 26){$\rightarrow$}
        \put(-85, 26){\small $-$}
        \put(-72, 50){\small $+$}
        \put(-72, 2){\small $+$}
        \put(-38, 26){\small $-$}
        \put(-20, 26){\small $+$}
        \caption{}
    \end{subfigure}
    \caption{(A) A positve Type I move. (B) A negative Type I move. (C) A positive Type II move. (D) A negative Type II move.}
    \label{fig: oriented reduction moves}
\end{figure}
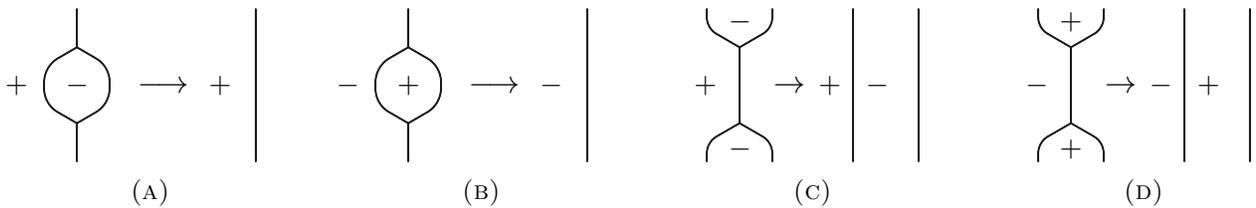

Forgetting orientations, we may consider $\mathcal{R}^\mu_\nu$ as a subset of $\mathcal{R}^m_n$. In particular, $\mathcal{R}^\mu_\mu$ can be regarded as a subgroup of $\mathcal{R}^m_m$. In Section \ref{sec:strand diagrams} we established an isomorphism between $F$ and $\mathcal{R}^{m}_{m}$, for each $m$. Next we will consider an analogue for  $\mathcal{R}^\mu_\mu$.

Jones introduced the subgroup $\vec{F} \leq F$, defined as elements $g \in F$ for which the resulting link diagram bounds an orientable surface, obtained by checkerboard coloring, see Figure \ref{fig: surface}. Fixing the leftmost region of this surface to be positive, every element of $\vec{F}$ produces an oriented link, whose orientation is induced by that of the surface it bounds \cite{jones14}. 
\begin{center}
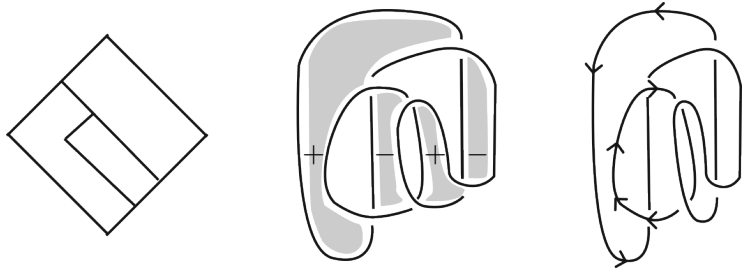
\begin{figure}[H]
\includegraphics[scale=0.35]{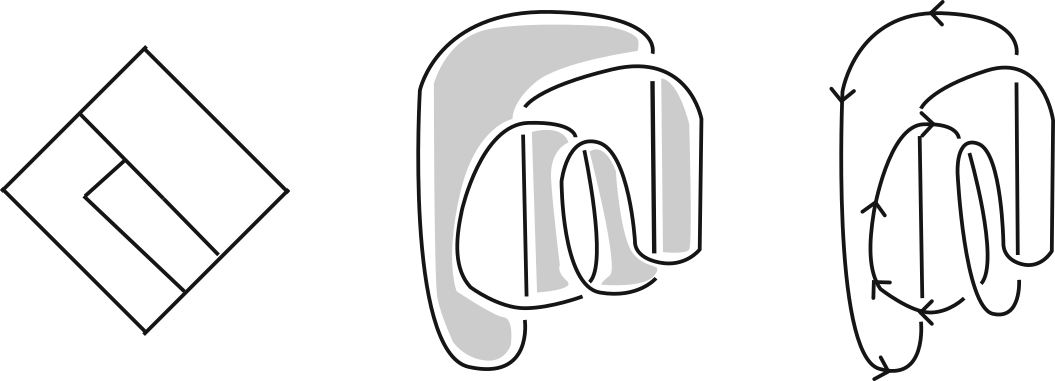}
{\small 
\put(-166,40){$+$}
\put(-139,40){$-$}
\put(-120,40){$+$}
\put(-104,40){$-$}
}
\caption{An oriented link built from an element of $\vec{F}$.}
\end{figure} \label{fig: surface}
\end{center}

As a key step toward generalizing this to build oriented $(n,n)$-tangles from $\vec{F}$, we now discuss isomorphisms between $\vec{F}$ and groups of oriented strand diagrams. Starting with the previously discussed isomorphism $\Theta_{0}$, observe that for $g \in \vec{F}$ each region of the complement of $\Theta_{0}(g)$ corresponds to a unique shaded region of the checkerboard surface resulting from $g$. This can be used to produce an orientation on $\Theta_{0}(g)$, by letting regions of the complement of $\Theta_{0}(g)$ inherit the signs of their associated regions of the checkerboard surface. Due to the convention that the leftmost region is always assigned $+$, the orientation on $\Theta_{0}(g)$ results in an element of $\mathcal{R}^{+}_{+}.$ 
Here $\mathcal{R}^{+}_{+}$ stands for $\mathcal{R}^\mu_\mu$ where $\mu$ is the $1$-sign $+$.
We now show that this restriction $\vec{\Theta}_{0}:=\Theta_{0}|_{\vec{F}}\colon\vec{F} \to \mathcal{R}^{+}_{+}$ is surjective, and therefore an isomorphism. 

\begin{prop}
    The map $\vec{\Theta}_{0}\colon \vec{F} \to \mathcal{R}_+^+$ is an isomorphism.
\end{prop}

\begin{proof}
    We just need to show $\vec{\Theta}_{0}$ is surjective.

    Indeed, for any $\vec{\Gamma} \in \mathcal{R}^+_+$, we can decompose $\vec{\Gamma}$ into a tree $\vec{s}$ and an inverse tree $(\vec{t})^*$ satisfying $\sigma_s = \sigma_t$. Here, given an oriented tree $\vec s$ with $n$ leaves, $\sigma_s$ denotes the $n$-sign induced by the orientation of $\vec s$. 
    Let $g$ be the element in $F$ corresponding to the tree diagram $(s, t)$. Then $g \in \vec{F}$ and $\vec{\Gamma} = \vec{\Theta}_{0}(g)$.
\end{proof}

For any oriented tree $\vec s$, we have an isomorphism $\vec{\Phi}_s\colon \mathcal{R}^+_+ \to \mathcal{R}^{\sigma_s}_{\sigma_s}$ sending $\vec{\Gamma}$ to its conjugate $\vec{s} * \vec{\Gamma} * (\vec{s})^*$. The maps $\vec{\Phi}_s$ are analogous to the isomorphisms $ \mathcal{R}^1_1 \to \mathcal{R}^n_n$, considered in the unoriented case in the paragraph preceding Definition \ref{def: gammakg}.
The inverse $(\vec{\Phi}_s)^{-1}$ uses conjugation with $(\vec{s})^*$ instead.

\begin{cor} \label{cor: oriented Thompson group and some oriented strand diagram groups are isomorphic}
    Given any $n\in {\mathbb N}$ and an $n$-sign $\nu$, we have $\vec{F} \cong \mathcal{R}^{\nu}_{\nu}$.
\end{cor}

\begin{proof}
   If $\nu$ is an $n$-sign, we can find a tree $\vec{s}$ such that $\nu = \sigma_s$ (\cite{jones14}, Section 5.2). Then $\vec{\Phi}_s \circ \vec{\Theta}_0\colon \vec{F} \to \mathcal{R}^{\sigma_s}_{\sigma_s} = \mathcal{R}^\nu_\nu$ is the desired isomorphism.
\end{proof}

Therefore, for any $n$-sign $\nu$, we have a nonempty family of isomorphisms $\{\vec{\Gamma}_s = \vec{\Phi}_s \circ \vec{\Theta}_0\}_{\sigma_s = \nu}$ from $\vec{F}$ to $\mathcal{R}^\nu_\nu$. Note that the oriented strand diagram $\vec{\Gamma}_s(g)$ is the unoriented strand diagram $\Gamma_s(g)$ with an orientation on it.

\begin{rem}
The map $\vec{\Phi}_s\colon \mathcal{R}^+_+ \to \mathcal{R}^{\sigma_s}_{\sigma_s}$ is given by conjugation by an oriented tree $\vec{s}$. More generally, we can take conjugate by a positively oriented $(1, n)$-{\em strand diagram} $\vec{\gamma}$, where $\gamma$ is an orientable $(1, n)$-strand diagram and we let $\vec{\gamma}$ have positive orientation.

    Similarly to the case of an oriented tree $\vec{s}$, an oriented $(1, n)$-strand diagram $\vec{\gamma}$ also induces a sequence of $+$'s and $-$'s, denoted as $\sigma_{\gamma}$. However, it is possible that the second sign in this sequence is not $-$. We call such a sequence a generalized $n$-sign. It turns out that every generalized $n$-sign can be induced by some $\gamma$, so we can follow the same strategy to prove $\vec{F} \cong \mathcal{R}_\nu^\nu$ for any generalized $n$-sign $\nu$.
\end{rem}

As in the unoriented case, we will always take $s$ to be the symmetric tree $S_k$ for some positive integer $k$. It turns out that the $2^k$-sign induced by $S_k$ has a special form. Recall that for any $n$-sign $\nu$ and positive integer $k$, we introduced a $(2^k n)$-sign $2^k \nu$ in Definition \ref{def: n sign}.

\begin{prop}
    Given a positive integer $k$, the $2^k$-sign $\sigma_{S_k}$ induced by symmetric tree $S_k$ is $2^k +$, where $+$ is considered as a $1$-sign.
\end{prop}

\begin{proof}
    We give a proof by induction. When $k = 1$, $S_1$ has a single split. After we put a $+$-sign in the leftmost region, this split becomes a positive split. By Figure \ref{fig: positive split}, a positive split induces $(+, -)$, so the $2$-sign induced by $S_1$ is $(+, -) = 2+$.

    Assume $S_k$ induces the $2^k$-sign $\sigma_{S_k} = 2^k +$. Then notice that to get $S_{k+1}$, we just need to add a split at each leaf of $S_k$. The split at the $i^{th}$ leaf of $S_k$ is a positive (resp. negative) split if the $i^{th}$ sign in $\sigma_{S_k}$ is $+$ (resp. $-$), because this sign is in the upper left region of the split. By Figure \ref{fig: positive split} and Figure \ref{fig: negative split}, a positive split induces $(..., +, -, ...)$, a negative split induces $(..., -, +,...)$. Thus, to obtain $\sigma_{S_{k + 1}}$, we just need to replace each $(...,+,...)$ in $\sigma_{S_k}$ by $(...,+, -,...)$, each $(...,-,...)$ in $\sigma_{S_{k}}$ by $(...,-,+,...)$. It follows that $\sigma_{S_{k+1}} = 2\sigma_{S_{k}} = 2 (2^k +) = 2^{k+1} +$.
\end{proof}

Thus for any positive integer $k$, we have $\vec{F} \cong \mathcal{R}^{2^k +}_{2^k +}$, with an explicit isomorphism given by $\vec{\Theta}_{k} := \vec{\Phi}_{S_k} \circ \vec{\Theta}_0$. 

\subsection{From oriented strand diagrams to oriented tangles}

Let $S$ be an oriented $(p, q)$-tangle in the square (with $p$ points on the top side and $q$ points on the bottom side of the square), and consider the induced upward or downward orientations at its endpoints. On both the top and the bottom, we indicate the downward orientation of an endpoint with a $+$, and the upward orientation with a $-$ sign. Note that the signs are positioned at the endpoints of the tangle; this is different from the case of strand diagrams where signs denote orientations of the subintervals on the top and the bottom of the square, separated by the endpoints. If $S$ induces a $p$-sign $\zeta$ on the top, and a $q$-sign $\eta$ on the bottom side, we say that $S$ is an oriented tangle from $\zeta$ to $\eta$.

\begin{defn}
    Let $\vec{\mathfrak{D}}$ be the category with objects $n$-signs for arbitrary positive integers $n$. The morphisms from $\mu$ to $\nu$ are oriented strand diagrams from $\mu$ to $\nu$. Compositions are concatenations of strand diagrams and signed regions.

    Similarly, let $\vec{\mathfrak{T}}$ be the category with objects $n$-signs for arbitrary positive integers $n$. The morphisms from $\mu$ to $\nu$ are oriented tangles from $2\mu$ to $2\nu$. Compositions are concatenations of tangles.
\end{defn}

Similarly to the functor $T\colon \mathfrak{D} \to \mathfrak{T}$, we construct a functor $\vec{T}\colon \vec{\mathfrak{D}} \to \vec{\mathfrak{T}}$ sending oriented strand diagrams to oriented tangles.  $\vec{T}$ acts by the identity on objects. Then suppose that we have an oriented strand diagram $\vec{\Gamma}$ from an $m$-sign $\mu$ to an $n$-sign $\nu$. Denote by $\Gamma$ the unoriented strand diagram underlying $\vec{\Gamma}$. We will construct $\vec{T}(\vec{\Gamma})$ by assigning an orientation to the unoriented tangle $T(\Gamma)$.

Suppose $P_1, P_2, ..., P_N$ are signed polygons in the complement of $\vec{\Gamma}$, excluding the rightmost one. Recall that to construct $T(\Gamma)$, we add an edge in each polygon $P_i$ to get a graph $\Gamma'$, then replace each $4$-valent vertex by a crossing as in Figure \ref{fig: 4-valent vertices to oriented crossings}. Notice that each time we add an edge, some signed polygon $P_i$ is split into two regions. To obtain an orientation on $T(\Gamma)$, we move the sign of each $P_i$ to the region to the right of the added edge. Now consider the induced orientation of the boundary arcs of each signed region.

The orientation of $\vec{T}(\vec{\Gamma})$ around each crossing can also be obtained by a local replacement at $4$-valent vertices of $\Gamma'$ as in Figure \ref{fig: 4-valent vertices to oriented crossings}. So the induced orientations of arcs are compatible around crossings.

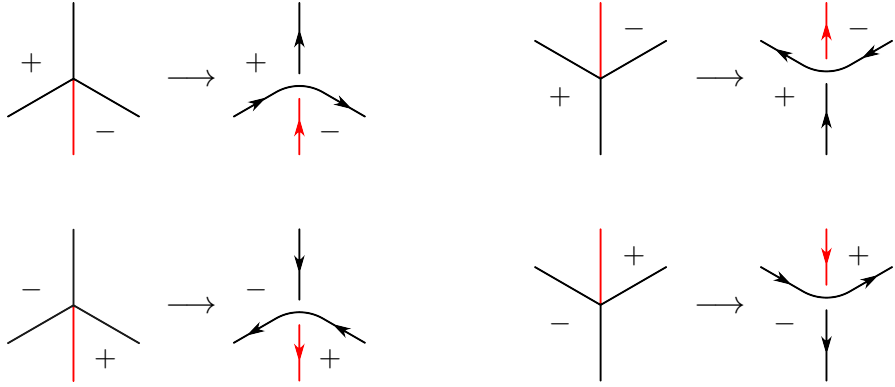
\begin{figure}
    \centering    
    \includegraphics{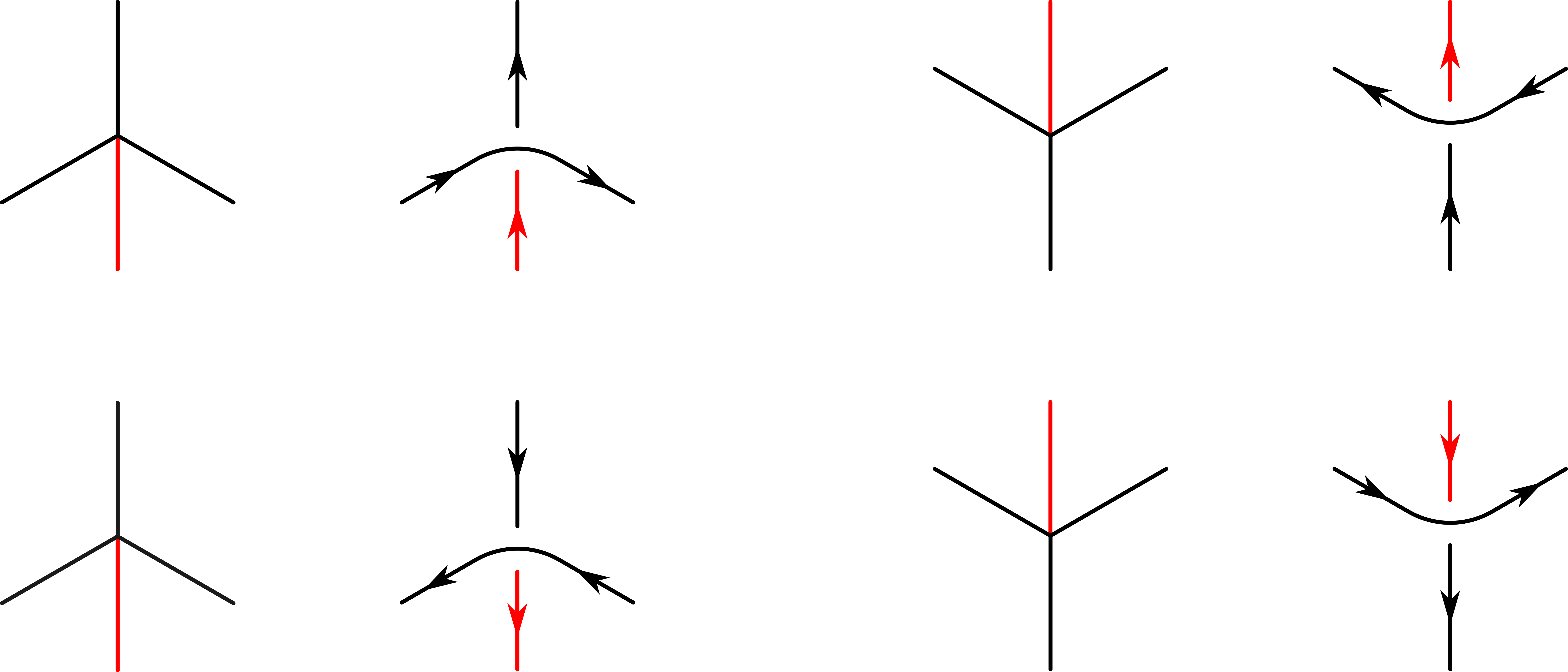}
    \put(-275, 112){$\longrightarrow$}
    \put(-275, 26){$\longrightarrow$}
    \put(-75, 112){$\longrightarrow$}
    \put(-75, 26){$\longrightarrow$}
    \put(-330, 118){\small $+$}
    \put(-302, 92){\small $-$}
    \put(-245, 118){\small $+$}
    \put(-217, 92){\small $-$}
    \put(-330, 32){\small $-$}
    \put(-302, 6){\small $+$}
    \put(-245, 32){\small $-$}
    \put(-217, 6){\small $+$}
    \put(-130, 105){\small $+$}
    \put(-102, 131){\small $-$}
    \put(-45, 105){\small $+$}
    \put(-17, 131){\small $-$}
    \put(-130, 19){\small $-$}
    \put(-102, 45){\small $+$}
    \put(-45, 19){\small $-$}
    \put(-17, 45){\small $+$}
    \caption{Around each crossing of $\vec{T}(\vec{\Gamma})$, its orientation can be equivalently obtained by replacing $4$-valent vertices of $\Gamma'$ by oriented crossings.}
    \label{fig: 4-valent vertices to oriented crossings}
\end{figure}

Note that each $+$-signed region induces the downward ($+$) orientation on its left boundary arc, and the upward ($-$) orientation on its right boundary arcs. The induced orientations are reversed for a $-$ signed region. So if $\vec{\Gamma}$ is an oriented strand diagram from $\mu$ to $\nu$, then $\vec{T}(\vec{\Gamma})$ is indeed an oriented strand diagram from $2\mu$ to $2\nu$ (see Figure \ref{fig: oriented strand diagram to tangle example}).

\begin{figure}
    \centering
    \includegraphics{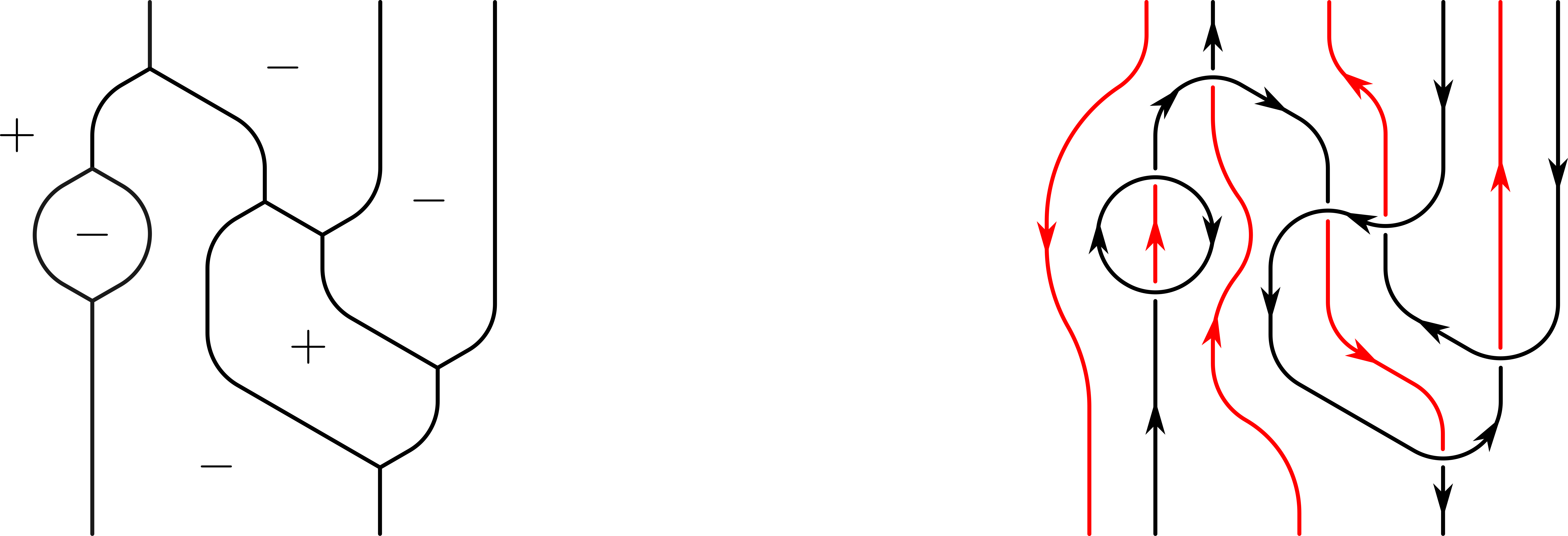}
    \put(-280,-20){$\vec{\Gamma}$}
    \put(-70, -20){$\vec{T}(\vec{\Gamma})$}
    \caption{An oriented strand diagram $\vec{\Gamma}$ from $(+, -, -)$ to $(+, -)$ and its associated oriented tangle $\vec{T}(\vec{\Gamma})$ from $(+, -, -, +, -, +)$ to $(+, -, -, +)$.}
    \label{fig: oriented strand diagram to tangle example}
\end{figure}

Similarly to the asymptotic faithfulness of the construction of unoriented tangles from $F$, we have an oriented counterpart as follows.

\begin{thm}
	Suppose $g, h \in \vec{F}, g \neq h$. Let $n = \max({\rm height}(g), {\rm height}(h))$. For $k$ such that $2^k \geq n$, we have $\vec{T}(\vec{\Theta}_k(g)) \not\cong \vec{T}(\vec{\Theta}_k(h))$.
\end{thm}

The proof is immediate, since $T(\Theta_k(g)) \neq T(\Theta_k(h))$, where $\Theta_k(g), \Theta_k(h)$ denote the unoriented strand diagrams underlying 
$\vec\Theta_k(g), \vec\Theta_k(h)$.

\section{A lax group action on Khovanov's chain complexes} \label{sec: Khovanov}
\subsection{Definition of a lax group action}
We now introduce two bicategories involved in the definition of a lax group action. For the relevant background material, we refer the reader to \cite[Ch. $2$]{2catstext}. 
\begin{defn}
    Let $G$ be a group. Define $*_{G}$ to be the bicategory with the single object $\{*\}$ and $\text{Mor}(*, *) = G$, where composition of morphisms $g \circ h$ is given by the group operation $gh$ for all $g, h \in G$. We define the $2$-morphisms to be trivial, and the associator and unitors to be identity natural transformations. 
\end{defn}
\begin{rem}
    Note that $*_{G}$ is also a $2$-category, since its  associators and unitors are identity natural transformations. 
\end{rem} 
\begin{defn}
    Consider the $2$-category  $\mathcal{C}at$, whose objects are small categories, the morphisms are functors, and the $2$-morphisms are natural transformations between functors. 
\end{defn}

\begin{defn}\label{defn: lax group action} Given a group $G$ and a small category $\mathcal{C}$, a lax group action of $G$ on $\mathcal{C}$ is defined to be a lax functor $P\colon *_G \to \mathcal{C}at$ such that $P(*) = \mathcal{C}$.
\end{defn}

We refer the reader to \cite[Definition $4.1.2$]{2catstext} for the definition of a lax functor. From this definition and the properties of $\mathcal{C}at$ and of $*_{G}$, it follows that a lax group action is determined by the following data: 

\begin{enumerate} 
	\item[(a)] a family of functors $\{P_g\colon \mathcal{C} \to \mathcal{C} \mid g \in G\},$
	\item[(b)] a natural transformation $\lambda\colon Id \to P_e$, where $Id$ is the identity endofunctor of $\mathcal{C}$ and $e\in G$ is the identity element,
	\item[(c)] a family of natural transformations $\{t_{h, g}\colon P_h \circ P_g \to P_{hg} \mid g, h \in G\}$,
\end{enumerate} satisfying
\begin{enumerate}
	\item For any $g \in G$, any $A \in \mathcal{C}$, we have $t_{e, g}(A) \circ \lambda_{P_g(A)} = {\rm Id}_{P_{g}(A)}$ and $t_{g, e}(A) \circ P_g(\lambda_A) = {\rm Id}_{P_{g}(A)}$, 
	\item For any $g, h, k \in G$, any $A \in \mathcal{C}$, the following diagram commutes
	\[\begin{tikzcd}
		& {P_k \circ P_h \circ P_g(A)} \\
		{P_{kh} \circ P_g(A)} && {P_k \circ P_{hg}(A)} \\
		& {P_{khg}(A)}
		\arrow["{t_{k, h}(P_g(A))}"', from=1-2, to=2-1]
		\arrow["{t_{kh, g}(A)}"', from=2-1, to=3-2]
		\arrow["{P_k(t_{h, g}(A))}", from=1-2, to=2-3]
		\arrow["{t_{k, hg}(A)}", from=2-3, to=3-2]
	\end{tikzcd}\]	
\end{enumerate}

\begin{rem}
    A genuine group action of $G$ on $\mathcal{C}$ is a pseudo functor (cf. \cite[Definition $4.1.2$]{2catstext}) $Q\colon *_G \to \mathcal{C}at$ such that $Q(*) = \mathcal{C}$, and $Q_g \in \mathcal{M}or(\mathcal{C}, \mathcal{C})$ is an autoequivalence for any $g \in G$. In other words, a lax group action $P$ is a genuine group action if $P$ further satisfies

    \begin{enumerate}
        \item[(3)] $P_g\colon \mathcal{C} \to \mathcal{C}$ is an autoequivalence for any $g$.
        \item[(4)] $\lambda\colon Id \to P_e$ is a natural isomorphism.
        \item[(5)] $t_{h, g}\colon P_h \circ P_g \to P_{hg}$ is a natural isomorphism for any $g, h \in G$.
    \end{enumerate}

\end{rem}

\subsection{Khovanov Chain Complexes}\label{subsec: Khovanov}  We assume the reader is familiar with Khovanov's functor-valued invariant of tangles, specifically, the construction of the ring $H^{n}$ \cite[Section $2.4$]{khovanov} and the chain complex of geometric $(H^{m},H^{n})$-bimodules associated to each $(m,n)$-tangle \cite[Section $3.4$]{khovanov}. Following Khovanov's notation, for a tangle $T$ we refer to its Khovanov chain complex as $\mathcal{F}(T)$. 

Recall that an $(H^{m},H^{n})$-bimodule is called \textit{geometric} if it is isomorphic to a finite direct sum of bimodules $\mathcal{F}(a)$, where $a$ are flat tangles with $2n$ boundary points on the lower edge, and $2m$ boundary points on the upper edge \cite[Section $2.8$]{khovanov}. We use the notation of \cite[Section $2.8$]{khovanov} and let $\mathcal{K}^{m}_{n}$ refer to the category of bounded chain complexes of geometric $(H^{m},H^{n})$-bimodules, up to chain homotopy.
We will consider the homotopy category $\overline{\mathcal{K}}^{m}_{n}$ of (not necessarily bounded) chain complexes of $(H^{m},H^{n})$-bimodules which are countable direct sums of bimodules $\mathcal{F}(a)$. 
\begin{rem}
We note that some of the usual features of the category ${\mathcal{K}}^{m}_{n}$ are not available for $\overline{\mathcal{K}}^{m}_{n}$; for example its Grothendieck group is undefined. This triangulated category will be the setting of the action of the Thompson group constructed below. In fact, the chain complexes associated with strand diagrams in Definition \ref{def: chain complex} are just a mild generalization of Khovanov's invariant of tangles, see Remark \ref{rem: chain complex}.
\end{rem}

In \cite{khovanov2}, Khovanov defined a 2-category $\mathcal{C}$ to be a combinatorial realization of the 2-category of tangle cobordisms. $\mathcal{C}$ has the same objects and 1-morphisms as $\vec{\mathfrak{T}}$. The 2-morphisms of $\mathcal{C}$ are ``movies" of tangle diagrams describing tangle cobordisms, up to Carter-Saito's movie moves. He also defined a 2-category of geometric bimodule complexes $\widehat{\mathbb{K}}$, with objects nonnegative integers, 1-morphisms $Mor(m, n) = \mathcal{K}^m_n$ and $2$-morphisms defined to be all (not necessarily grading preserving) morphisms of complexes of bimodules up to chain homotopy and up to a sign. A well-defined 2-functor $\mathcal{F}: \mathcal{C} \to \widehat{\mathbb{K}}$ is constructed in \cite[Section 4]{khovanov2}. 

\begin{rem}\label{rem: ring}
    One can define a ring similar to Khovanov's ring $H^n$, using tangles associated with strand diagrams. Here a direct sum is taken over strand diagrams, as opposed to the direct sum over flat tangles in \cite{khovanov}. The multiplication in this ring may be thought of as being modeled on the multiplication in Thompson's group $F$. We will not pursue this construction in the present paper.
\end{rem}

\subsection{Construction of the lax group action} We begin by introducing, for each reduced $(m,n)$-strand diagram $D$, a chain complex $C^{*}(D) \in \text{Ob}(\overline{\mathcal{K}}^{m}_{n})$. Recall that every reduced oriented strand diagram is equal to $F_{1}^{*} \circ F_{2}$ for some pair of oriented forests $F_{1}, F_{2}$. In this section we assume that all strand diagrams and tangles are oriented, and for brevity of notation we will omit arrows on symbols denoting oriented tangles and strand diagrams.

\begin{defn} Let $D=F_{1}^{*} \circ F_{2}$ be a reduced oriented $(m,n)$-strand diagram, and let $\ell \in \mathbb{N}$ be the number of leaves in $F_{1}$ and $F_{2}$. If $L$ is some forest with $\ell$ roots whose orientation is compatible with that of $F_{1}$ and $F_{2}$, we can create a new (not necessarily reduced) strand diagram $D(L)$ given by
\[D(L):=F_{1}^{*} \circ L^{*} \circ L \circ F_{2}.\]
\end{defn}
 
\begin{defn} \label{def: chain complex} Let $D=F_{1}^{*} \circ F_{2} \in \vec{\mathcal{R}^{\mu}_{\nu}}$ be a reduced oriented strand diagram, where $\mu$ is an $m$-sign and $\nu$ is an $n$-sign, and let $\ell$ be as above. We use $F(\ell)$ to refer to the set of all oriented forests with $\ell$ roots and orientations compatible with that of $F_{1},F_{2}$. Then we define
\[ C^{*}(D):=\bigoplus_{L \in F(\ell)}\mathcal{F}(\vec{T}(D(L))) \in \overline{\mathcal{K}}^{m}_{n},\]
where $\vec{T}$ is the functor defined in Section \ref{sec: orientability} and $\mathcal{F}(\vec{T}(D(L)))$ is the Khovanov chain complex.
\end{defn}

\begin{rem}\label{rem: chain complex}
The strand diagram $D(L)$ is equivalent to $D$; in fact $L^*\circ L$ may be reduced to the identity strand diagram using type I moves. Thus the chain complex $C^*(D)$ is assembled of Khovanov complexes of tangles corresponding to all ways of stabilizing $D$ using the type I move. The reason for this definition will be clear in the construction of the lax group action below; specifically, Proposition \ref{prop: tangle cobordism is independent of Type II move sequence} shows that the chain map corresponding to a composition of type II reductions is well-defined but this is not necessarily the case for a composition of both type I and type II reductions. Also note that the effect on the tangle $\vec{T}(D)$ of a stabilization by $L^*\circ L$ is an introduction of unknots which may be isotoped off of the tangle diagram for $\vec{T}(D)$ using Reidemeister II moves. Therefore up to chain homotopy, each summand 
$\mathcal{F}(\vec{T}(D(L)))$ is a direct sum of several copies of $\mathcal{F}(\vec{T}(D))$ with $q$-degree shifts.
\end{rem}

We now introduce the first bit of data required to determine a lax group action, that is, a family of functors $\{P_g\colon \overline{\mathcal{K}}^{2^{k}}_{n}\to \overline{\mathcal{K}}^{2^{k}}_{n}\mid g \in \vec{F}\}$.  

\begin{defn} \label{def: unitor of lax Thomospon group action}
Fix $k \geq 0, n \geq 0$. Let $A^{*},B^{*}$ be chain complexes in $\overline{\mathcal{K}}^{2^k}_{2^k}$, and let $f\colon A^{*} \to B^{*}$ be a chain map. For each $g \in \vec{F}$ define an  endofunctor $P_{g}$ of $\overline{\mathcal{K}}^{2^{k}}_{n}$ by 
 \[P_{g}(A^{*}):=C^{*}(\Theta_{k}(g)) \otimes_{H^{2^k}} A^{*},\] and \[P_{g}(f):={\rm Id} \otimes f\colon C^{*}(\Theta_{k}(g)) \otimes_{H^{2^k}} A^{*} \to C^{*}(\Theta_{k}(g)) \otimes_{H^{2^k}} B^{*}.\]
\end{defn}

Next we introduce the natural transformation $\lambda$ from the identity endofunctor 
of $\overline{\mathcal{K}}^{2^{k}}_{n}$
to $P_{e}$, the functor associated to the identity element $e \in \vec{F}$ as above. This consists of, for each $A^* \in \overline{\mathcal{K}}^{2^k}_{n}$, a chain map $\lambda_{A^*}\colon A^* \to C^{*}(\Theta_{k}(e))\otimes_{H^{2^k}} A^*$.

By definition, $C^{*}(\Theta_{k}(e)) = \bigoplus_{L \in F(2^k)} \mathcal{F}(\vec{T}(\Theta_k(e)(L)))$. Consider the direct summand given by letting $L$ be ${\rm Id}_{2^k}$, the identity $(2^k+, 2^k+)$-strand diagram. It is also true that $\Theta_{k}(e)={\rm Id}_{2^{k}}$. The functor $\vec{T}$ sends $\Theta_k(e)$ to the identity $(2^{k + 1} +, 2^{k + 1} +)$-tangle $\textup{Vert}_{2^{k+1}}$, by introducing an extra strand to the left of each component. Therefore
$$
\mathcal{F}(\vec{T}(\Theta_k(e)({\rm Id}_{2^k}))) = \mathcal{F}(\vec{T}(\Theta_k(e))) = \mathcal{F}(\textup{Vert}_{2^{k+1}})
$$
is a chain complex with a single homological grading. Note that $\mathcal{F}(\textup{Vert}_{2^{k+1}}) = H^{2^k}$ as $H^{2^k}$-bimodules \cite{khovanov}. We denote the unit in $H^{2^k}$ by $1_{2^k}$.

\begin{defn} For each $A^* \in \overline{\mathcal{K}}^{2^{k}}_{n}$ the natural transformation 
\[\lambda_{A^*}\colon A^* \to C^{*}(\Theta_{k} (e))\otimes_{H^{2^k}} A^* \] from the identity endofunctor to $P_e$ is given by tensoring with $1_{2^{k}}$, i.e. taking $A^*$ by the identity map to the direct summand $1_{2^{k}} \otimes A^*$. 
\end{defn}
In order to give a family of natural transformations $\{t_{h, g}\colon P_h \circ P_g \to P_{hg} \mid g, h \in \vec{F}\},$ we introduce some preliminary definitions and lemmas.

\begin{defn}
    Let the 2-category $\mathcal{D}$ have the same objects and 1-morphisms as $\vec{\mathfrak{D}}$. Define the 2-morphisms in $\mathcal{D}$ to be sequences of Type II moves (see Definition \ref{def: oriented reduction moves}).
\end{defn}

Recall Khovanov's $2$-category ${\mathcal C}$, mentioned at the end of Section \ref{subsec: Khovanov}. We extend the 1-functor $\vec{T}\colon \vec{\mathfrak{D}} \to \vec{\mathfrak{T}}$ to a 2-functor $\mathbb{T}: \mathcal{D} \to \mathcal{C}$, by keeping its action on objects and on 1-morphisms, and defining its action on 2-morphisms of $\mathcal{D}$ as follows. We let $\mathbb{T}$ send a positive Type II move to the movie shown in Figure \ref{fig: positive type II movie}, which is a saddle cobordism followed by a Reidemeister II move. Reversing orientations, we obtain the movie of a negative Type II move, see Figure \ref{fig: negative type II movie}.

\begin{figure}[ht]
    \centering
    \begin{subfigure}{\textwidth}
        \centering
        \includegraphics{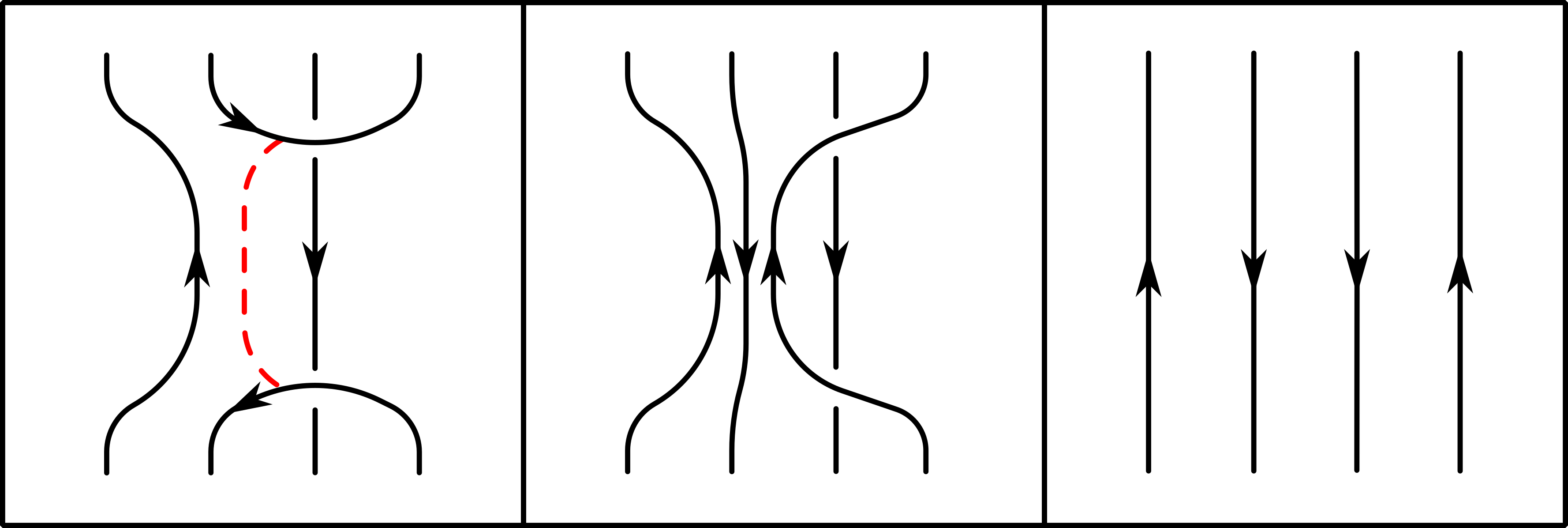}
        \caption{}
        \label{fig: positive type II movie}
    \end{subfigure}
    \vspace{3mm}
    \vfill
    \begin{subfigure}{\textwidth}
        \centering
        \includegraphics{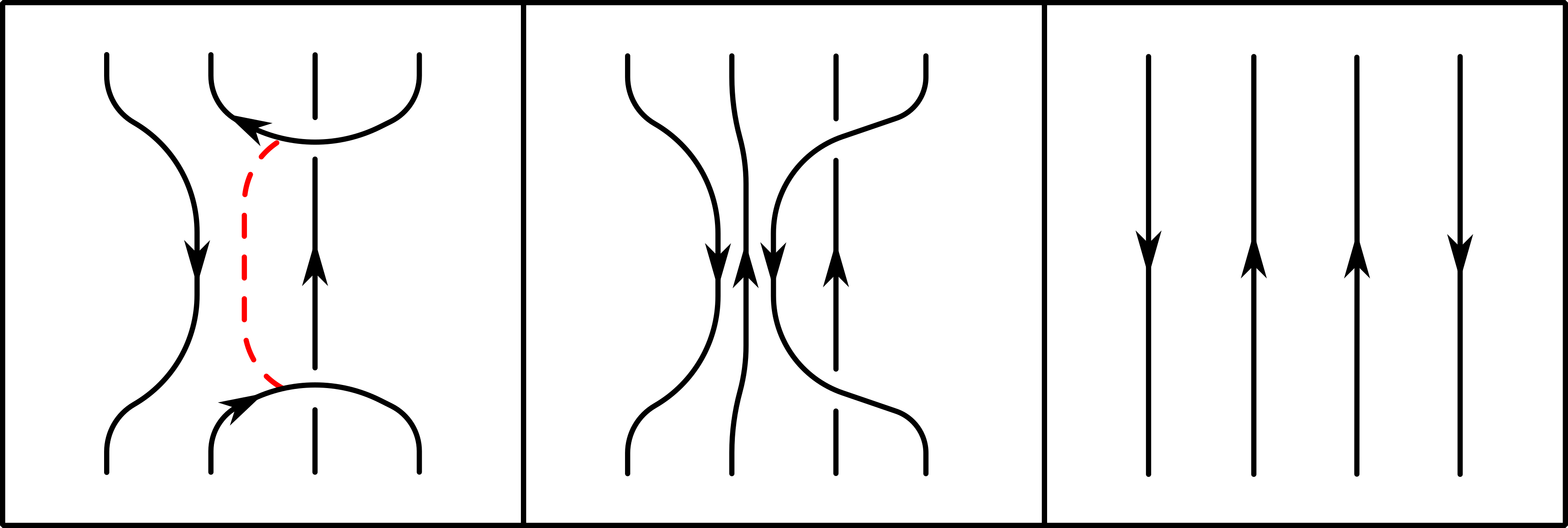}
        \caption{}
        \label{fig: negative type II movie}
    \end{subfigure}
    \caption{Movies of (A) positive Type II move. (B) negative Type II move.}
\end{figure}

\begin{defn}
    An oriented strand diagram is said to be {\em Type II-reduced} if it is not subject to any Type II moves.
\end{defn}

Given a $(\mu, \nu)$-strand diagram $\Gamma$, let $r$ denote a sequence of Type II moves taking $\Gamma$ to $\Gamma'$, where $\Gamma'$ is Type II-reduced. By an
argument similar to the proof of \cite[Proposition 2.1.1]{matucci}, $\Gamma'$ is unique. However the choice of a sequence $r$ from $\Gamma$ to $\Gamma'$ is not unique in general.
Note that each Type II move in $r$ corresponds to a pair of vertices, the endpoints of the edge which is removed.

\begin{prop}  \label{prop: tangle cobordism is independent of Type II move sequence}
    $\mathbb{T}(r)$ is independent of a choice of $r$.
\end{prop} 

\begin{proof}
To prove the statement, we will characterize pairs of vertices in $\Gamma$ encoding the Type II moves in a sequence reducing $\Gamma$ to $\Gamma'$. To this end, consider paths (geodesics) in $\Gamma$ which are monotonic with respect to the height function, parametrized so that the height is a decreasing function of the parameter. 
It is convenient to consider paths which start and end at mid-points of edges of the strand diagram.
Any such path $\gamma$ follows a sequence of merges and splits. In more detail, consider a sequence of labels corresponding to each path $\gamma$: label it with $l$ (respectively $r$) whenever it enters a merge from the left (respectively from the right), and similarly label it with $l^{-1}$, $r^{-1}$ whenever it exits a split  on the left or on the right. 
For example, the path shown in Figure \ref{fig: admissible path} is labeled $lll^{-1}rr^{-1} l^{-1}$. Consider the free group $F_{l,r}$ generated by $l,r$; then the sequence of labels may be thought of as a word $w_{\gamma}$ in the generators.
\begin{figure}[ht]
\begin{center}
\includegraphics[height=5cm]{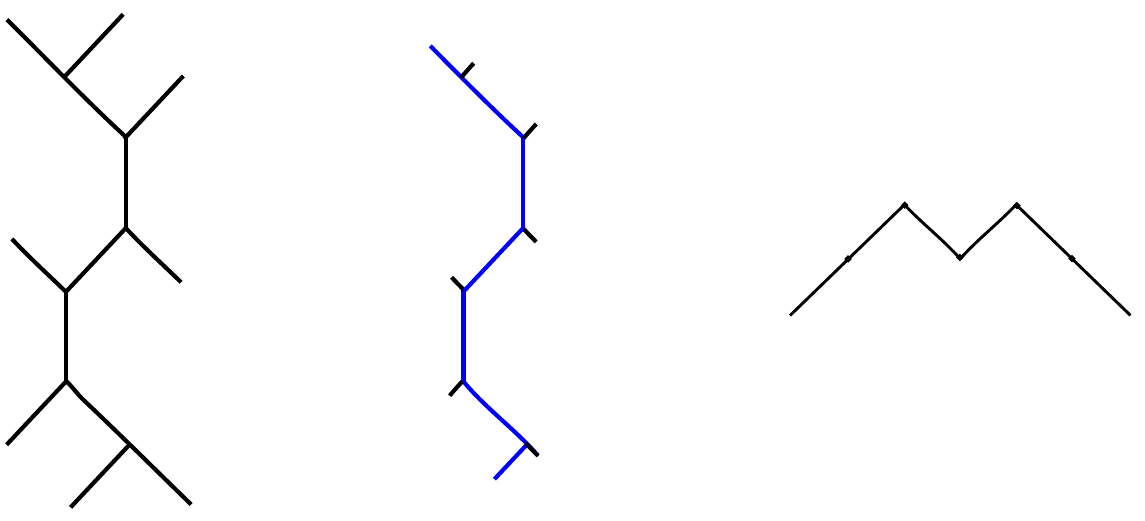}
{\scriptsize
\put(-195,117){$l$}
\put(-178,100){$l$}
\put(-184,80){$l^{-1}$}
\put(-180,60){$r$}
\put(-180,40){$r^{-1}$}
\put(-190,17){$l^{-1}$}
}
\caption{Left: a part of a strand diagram. Center: an admissible path $\gamma$, colored blue, with $w_{\gamma}=lll^{-1}rr^{-1} l^{-1}$. Right: the weight function corresponding to  $\gamma$.}\label{fig: admissible path}
\end{center}
\end{figure}

Now define integer weights labeling $\gamma$ as it passes mid-points of edges of the strand diagram. The beginning point of $\gamma$ is assigned weight $0$. The weight function is given by the augmentation of the word $w_{\gamma}$, that is the sum of exponents of the letters along the path, from the start to the given point. The weights may also be characterized by requiring that passing a merge adds $1$, passing a split split subtracts $1$.
Type II moves in a sequence $r$ reducing $\Gamma$ to $\Gamma'$ are in bijection with pairs of vertices in $\Gamma$ corresponding to {\em admissible} paths $\gamma$ satisfying the following properties:
\begin{enumerate}
    \item The exponent of the first letter in $w_{\gamma}$ is $+1$,
    \item The word $w_{\gamma}$ represents the trivial element of the free group $F_{l,r}$,
    \item The endpoint of $\gamma$ has weight $0$,
    \item The weights between the beginning and the end of $\gamma$ are strictly positive.
\end{enumerate}
An example of an admissible path is shown in Figure \ref{fig: admissible path}. The pair of vertices (specifying a type II move) corresponding to $\gamma$ are the first and the last vertices of $\Gamma$ it passes through, which are a merge and a split respectively, by properties (1), (3) and (4). Condition (2) corresponds to the fact that the ``vertical'' center interval (as shown in Figure \ref{fig: moves}) of each type II move in a sequence reducing $\Gamma$ to $\Gamma'$ is a result of previous type II moves in the sequence.

It follows from conditions (3), (4) that two admissible paths ${\gamma}_1, \gamma_2$ cannot be interlaced, meaning that ${\gamma}_1\cap \gamma_2$ contains precisely one endpoint of both ${\gamma}_1$ and $\gamma_2$.
In other words, any two admissible paths 
are either (a) disjoint, (b) nested, say ${\gamma}_1$ is in the interior of $\gamma_2$, or (c) they overlap, meaning that the intersection ${\gamma}_1\cap  \gamma_2$ is an interval contained in the interior of ${\gamma}_1$ and in the interior of $\gamma_2$. 
The saddle maps on Khovanov chain complexes arising from Type II moves along disjoint paths commute. The case (b) corresponds to the $\gamma_1$ move being done before the $\gamma_2$ move. In the case (c), a preceding move along $\gamma_1\cap \gamma_2$ makes $\gamma_1$  and $\gamma_2$ disjoint. In any of these case, either the order of moves is well-defined, or they commute, showing that the map on Khovanov chain complexes is independent of a choice of $r$. 
\end{proof}

\begin{rem}
The piece-wise linear graphs of admissible weight functions along admissible paths are precisely {\em Dyck paths} studied in combinatorics and statistical mechanics.
\end{rem}

Suppose $\Gamma$ is a Type II-reduced $(\sigma, \mu)$-strand diagram and $\Lambda$ is a Type II-reduced $(\mu, \rho)$-strand diagram. Let $\Lambda \cdot \Gamma$ denote the Type II-reduced $(\sigma, \rho)$-strand diagram corresponding to $\Lambda \circ \Gamma$. Note that only type II moves are used here, and in general $\Lambda \cdot \Gamma$ is different from the reduced strand diagram $\Lambda * \Gamma$ defined in Section \ref{sec:strand diagrams}, where both types I and II moves are used. Let $r$ be a sequence of type II moves reducing $\Lambda \circ \Gamma$ to $\Lambda \cdot \Gamma$.
Khovanov's functor $\mathcal{F}$ gives a chain map
$$
\mathcal{F}(\mathbb{T}(r))\colon \mathcal{F}(\mathbb{T}(\Lambda \circ\Gamma)) \to \mathcal{F}(\mathbb{T}(\Lambda \cdot \Gamma)).
$$

By \cite[Proposition 13]{khovanov2}, there is an isomorphism $\psi: \mathcal{F}(\mathbb{T}(\Lambda)) \otimes_{H^m} \mathcal{F}(\mathbb{T}(\Gamma))\to \mathcal{F}(\mathbb{T}(\Lambda) \circ \mathbb{T}(\Gamma)) = \mathcal{F}(\mathbb{T}(\Lambda \circ \Gamma))$. 

\begin{defn} Let $\Gamma,\Lambda$ be as above. Define the map
\begin{equation}
    \phi(\Gamma, \Lambda) := \mathcal{F}(\mathbb{T}(r)) \circ \psi: \mathcal{F}(\mathbb{T}(\Lambda)) \otimes_{H^m} \mathcal{F}(\mathbb{T}(\Gamma)) \to \mathcal{F}(\mathbb{T}(\Lambda \cdot \Gamma))
\end{equation}

By Proposition \ref{prop: tangle cobordism is independent of Type II move sequence}, $\mathbb{T}(r)$ is independent of a choice of $r$, so $\phi(\Gamma, \Lambda)$ is well defined.
\end{defn}

\vspace{5mm}

Let $g, h$ be any two elements in $\vec{F}$. Considering $\Theta_k(g), \Theta_k(h)$ and $\Theta_k(hg)$ as pairs of forests, let $u, v$ and $w$ denote their respective numbers of leaves.

\begin{lem}
   For any $U \in F(u)$ and $V \in F(v)$, there exists $W \in F(w)$ such that  \[\Theta_k(h)(V) \cdot \Theta_k(g)(U) = \Theta_k(hg)(W). \]
\end{lem}

\begin{proof}
    Letting $\sim$ denote the equivalence of strand diagrams, notice that \[\Theta_k(h)(V) \cdot \Theta_k(g)(U) \sim  \Theta_k(h)(V) \circ  \Theta_k(g)(U) \sim \Theta_k(h) \circ \Theta_k(g) \sim  \Theta_k(h) *  \Theta_k(g) =  \Theta_k(hg).\] Therefore $\Theta_k(h)(V) \cdot \Theta_k(g)(U)$ is a Type II-reduced strand diagram equivalent to the reduced strand diagram $ \Theta_k(hg)$. Next, observe that the type II-reduced strand diagram $ \Theta_k(h)(V) \cdot \Theta_k(g)(U)$ can be fully reduced by Type I moves.
    Said differently, $ \Theta_k(h)(V) \cdot \Theta_k(g)(U)$ can be obtained by a sequence of reversed Type I moves (inserting carets) from $\Theta_k(hg)$. 

    Moreover, as we explain next, these carets can only be put on certain edges. We know from the fact that $\Theta_{k}(hg)$ is reduced that $\Theta_k(hg) = E_2^* * E_1$ for some forests $E_1$ and $E_2$. We say an edge in $\Theta_k(hg)$ is a {\em bridge} if a concatenation point between $E^*_2$ and $E_1$ is on this edge. If a Type II-reduced graph differs from $E^{*}_{2} * E_{1}$ by carets, it can only be obtained from $E^{*}_{2} * E_{1}$ by inserting carets on \textit{bridges}, or on new bridges created by other previously inserted carets. Putting carets elsewhere creates a diagram that is not Type II-reduced, as such carets can always be cancelled in a Type II move. 
    
    The process of putting carets on bridges is equivalent to the process of inserting $W^*\circ W$ between $E^{*}_{2}$ and $E_{1}$, for some forest $W$. Therefore, $$\Theta_k(h)(V) \cdot \Theta_k(g)(U) = E_2^* \circ W^* \circ W \circ E_1 = \Theta_k(hg)(W)$$.
\end{proof}

\begin{defn} Given $g,h \in \vec{F}$, define
\begin{equation}
    \phi(h, g) = \bigoplus_{U\in F(u), V \in F(v)} \phi(\Theta_k(h)(V), \Theta_k(g)(U)),
\end{equation}
a chain map from \[C^*(\Theta_k(h)) \otimes_{H^{2^k}} C^*(\Theta_k(g)) = \bigoplus_{U \in F(u), V \in F(v)} \mathcal{F}(\mathbb{T}(\Theta_k(h)(V))) \otimes_{H^{2^k}} \mathcal{F}(\mathbb{T}(\Theta_k(g)(U)))\] to \[C^*(\Theta_k(hg)) = \bigoplus_{W \in F(w)} \mathcal{F}(\mathbb{T}(\Theta_k(hg)(W))).\]

Then for any $A^* \in \overline{\mathcal{K}}^{2^k}_n$, define
\begin{equation}
    t_{h, g}(A^*) = \phi(h, g) \otimes id
\end{equation}

from $P_h \circ P_g(A^*) = C^*(\Theta_k(h)) \otimes_{H^{2^k}} C^*(\Theta_k(g)) \otimes_{H^{2^k}} A^*$ to $P_{hg}(A^*) = C^*(\Theta_k(hg)) \otimes_{H^{2^k}} A^*$.
\end{defn}
\begin{thm} \label{thm: action}
    The data $(\{P_g\}_{g \in \vec{F}}, \lambda, \{t_{h, g}\}_{g, h \in \vec{F}})$ determines a lax action of $\vec{F}$ on $\overline{\mathcal{K}}^{2^k}_n$.
\end{thm}

Next, we verify that the collection of functors $\{P_g\}_{g \in \vec{F}}$, the natural transformation $\lambda$, and the collection of natural transformations $\{t_{h, g}\}_{g, h \in \vec{F}}$, satisfy the necessary properties outlined after Definition \ref{defn: lax group action}.

\subsection{Verifying the axioms} 

To check the first axiom of a lax group action, we need the following lemma.

\begin{lem}
    For any $g \in \vec{F}$ and any chain $u$ in $C^*(\Theta_k(g))$, we have $(\phi(g, e))(u \otimes 1^{2^k}) = u$ and $(\phi(e, g))(1^{2^k} \otimes u) = u$.
\end{lem}
\begin{proof}
    We will prove the first equation; a similar argument proves the second equation.

    By definition, $C^*(\Theta_k(g)) = \bigoplus_{L \in F(l)} \mathcal{F}(\mathbb{T}(\Theta_k(g)(L)))$, so we can assume $u \in \mathcal{F}(\mathbb{T}(\Theta_k(g)(L)))$ for some $L \in F(l)$ and check that $\phi(\Theta_k(g)(L), \Theta_k(e))(u \otimes 1^{2^k}) = u$.

    Notice that $\Theta_k(e)$ is a trivial strand diagram, so $\Theta_k(g)(L) \circ \Theta_k(e) = \Theta_k(g)(L)$ is Type II-reduced, i.e. the Type II move sequence $r$ from $\Theta_k(g)(L) \circ \Theta_k(e)$ to $\Theta_k(g)(L) \cdot \Theta_k(e)$ is empty. Thus $\mathcal{F}(\mathbb{T}(r))$ is identity map and
    $$
    \phi(\Theta_k(g)(L), \Theta_k(e)) = \mathcal{F}(\mathbb{T}(r)) \circ \psi = \psi
    $$

    Here $\psi$ is the isomorphism from $\mathcal{F}(\mathbb{T}(\Theta_k(g)(L))) \otimes_{H^{2^k}} \mathcal{F}(\mathbb{T}(\Theta_k(e))$ to $\mathcal{F}(\mathbb{T}(\Theta_k(g)(L) \circ \Theta_k(e))) = \mathcal{F}(\mathbb{T}(\Theta_k(g)(L)))$. To see $\psi(u \otimes 1^{2^k}) = u$, we just need to observe that $\psi$ is equivalent to the right $H^{2^k}$-action on chain modules of $\mathcal{F}(\mathbb{T}(\Theta_k(g)(L)))$. So the unit $1^{2^k}$ acts trivially.
\end{proof}

Using the above Lemma we may verify the first axiom. For any $u \otimes a \in P_g(A^*)$, we have
\begin{align*}
    &(t_{e, g}(A^*) \circ \lambda_{P_g(A^*)})(u \otimes a) = t_{e, g}(A^*)(1^{2^k} \otimes u \otimes a) = \phi(e, g)(1^{2^k} \otimes u) \otimes a= u \otimes a \\
    &(t_{g, e}(A^*) \circ P_g(\lambda_{A^*}))(u \otimes a) = t_{g, e}(A^*)(u \otimes 1^{2^k} \otimes a) = \phi(g, e)(u \otimes 1^{2^k}) \otimes a = u \otimes a
\end{align*}

So $t_{e, g}(A^*) \circ \lambda_{P_g(A^*)} = {\rm Id}_{P_g(A^*)}$ and $t_{g, e}(A^*) \circ P_g(\lambda_{A^*}) = {\rm Id}_{P_g(A^*)}$.

\vspace{0.5cm}

For the associativity axiom, we need to check that the following diagram commutes for any $f, g, h \in \vec{F}$.

\[\begin{tikzcd}
	{C^*(\Theta_k(h)) \otimes_{H^{2^k}} C^*(\Theta_k(g)) \otimes_{H^{2^k}} C^*(\Theta_k(f))} & {C^*(\Theta_k(h)) \otimes_{H^{2^k}} C^*(\Theta_k(gf))} \\
	{ C^*(\Theta_k(hg)) \otimes_{H^{2^k}} C^*(\Theta_k(f))} & {C^*(\Theta_k(hgf))}
	\arrow["{id \otimes \phi(g, f)}", from=1-1, to=1-2]
	\arrow["{\phi(h, g) \otimes id}"', from=1-1, to=2-1]
	\arrow["{\phi(hg, f)}"', from=2-1, to=2-2]
	\arrow["{\phi(h, gf)}", from=1-2, to=2-2]
\end{tikzcd}\]

Recall that $C^*(\Theta_k(g))$ is defined to be a direct sum of $\mathcal{F}(\mathbb{T}(\Theta_k(g)(L)))$ over forests $L$, where $\Theta_k(g)(L)$ is obtained by inserting $L^* L$ in the middle of $\Theta_k(g)$. The commutativity of the above diagram is therefore equivalent to the commutativity of the following diagram for any $ \Gamma, \Lambda, \Pi$ obtained by inserting forests in the middle of $\Theta_k(f), \Theta_k(g), \Theta_k(h)$ respectively.

\[\begin{tikzcd}
	{\mathcal{F}(\mathbb{T}(\Pi)) \otimes_{H^{2^k}} \mathcal{F}(\mathbb{T}(\Lambda)) \otimes_{H^{2^k}} \mathcal{F}(\mathbb{T}(\Gamma))} & {\mathcal{F}(\mathbb{T}(\Pi)) \otimes_{H^{2^k}} \mathcal{F}(\mathbb{T}(\Lambda \cdot \Gamma))} \\
	{\mathcal{F}(\mathbb{T}(\Pi \cdot \Lambda)) \otimes_{H^{2^k}} \mathcal{F}(\mathbb{T}(\Gamma))} & {\mathcal{F}(\mathbb{T}(\Pi \cdot \Lambda \cdot \Gamma))}
	\arrow["{id \otimes \phi(\Lambda, \Gamma)}", from=1-1, to=1-2]
	\arrow["{\phi(\Pi, \Lambda) \otimes id}"', from=1-1, to=2-1]
	\arrow["{\phi(\Pi \cdot \Lambda, \Gamma)}"', from=2-1, to=2-2]
	\arrow["{\phi(\Pi, \Lambda \cdot \Gamma)}", from=1-2, to=2-2]
\end{tikzcd}\]

In the above diagram, we choose a Type II move sequence $p$ from $\Pi \circ \Lambda$ to $\Pi \cdot \Lambda$, a sequence $q$ from $(\Pi \cdot \Lambda) \circ \Gamma$ to $\Pi \cdot \Lambda \cdot \Gamma$, a sequence $r$ from $\Lambda \circ \Gamma$ to $\Lambda \cdot \Gamma$, a sequence $s$ from $\Pi \circ (\Lambda \cdot \Gamma)$ to $\Pi \cdot \Lambda \cdot \Gamma$, then let  $\phi(\Pi, \Lambda) = \mathcal{F}(\mathbb{T}(p))$, $\phi(\Pi \cdot \Lambda, \Gamma) = \mathcal{F}(\mathbb{T}(q))$, $\phi(\Lambda, \Gamma) = \mathcal{F}(\mathbb{T}(r))$, $\phi(\Pi, \Lambda \cdot \Gamma) = \mathcal{F}(\mathbb{T}(s))$.

If we regard $p$ as a sequence of Type II moves from $\Pi \circ \Lambda \circ \Gamma$ to $(\Pi \cdot \Lambda) \circ \Gamma$ and $r$ as a sequence of Type II moves from $\Pi \circ\Lambda \circ \Gamma$ to   $\Pi \circ (\Lambda \cdot \Gamma)$, then we can rewrite the diagram as

\[\begin{tikzcd}
	& {\mathcal{F}(\mathbb{T}(\Pi \circ \Lambda \circ \Gamma))} \\
	{\mathcal{F}(\mathbb{T}((\Pi \cdot \Lambda) \circ \Gamma))} && {\mathcal{F}(\mathbb{T}(\Pi \circ (\Lambda \cdot \Gamma)))} \\
	& {\mathcal{F}(\mathbb{T}(\Pi \cdot \Lambda \cdot \Gamma))}
	\arrow["{\mathcal{F}(\mathbb{T}(p))}"', from=1-2, to=2-1]
	\arrow["{\mathcal{F}(\mathbb{T}(q))}"', from=2-1, to=3-2]
	\arrow["{\mathcal{F}(\mathbb{T}(r))}", from=1-2, to=2-3]
	\arrow["{\mathcal{F}(\mathbb{T}(s))}", from=2-3, to=3-2]
\end{tikzcd}\]

Note that $qp$ and $sr$ are two sequences of Type II moves from $\Pi \circ \Lambda \circ \Gamma$ to $\Pi \cdot \Lambda \cdot \Gamma$, so $\mathbb{T}(qp) = \mathbb{T}(sr)$ by Proposition \ref{prop: tangle cobordism is independent of Type II move sequence}. Then we have $\mathcal{F}(\mathbb{T}(q)) \circ \mathcal{F}(\mathbb{T}(p)) = \mathcal{F}(\mathbb{T}(qp)) = \mathcal{F}(\mathbb{T}(sr)) = \mathcal{F}(\mathbb{T}(s)) \circ \mathcal{F}(\mathbb{T}(r))$.

This concludes the proof of the associativity axiom and of Theorem \ref{thm: action}.

\bibliographystyle{plain}
\bibliography{refs}

\end{document}